\documentclass[11pt, reqno]{amsart}

\usepackage[noend]{algpseudocode}
\usepackage{algorithmicx,algorithm}
\usepackage{fullpage}

\usepackage[colorlinks,
linkcolor=red,
anchorcolor=blue,
citecolor=green
]{hyperref}
\usepackage{cite}

\usepackage{bm}
\usepackage{amsmath}
\usepackage{amssymb}
\usepackage{amsfonts}
\usepackage{amsthm}
\usepackage{mathrsfs}
\theoremstyle{plain} \newtheorem{thm}{Theorem}[section]
\theoremstyle{plain} \newtheorem{lem}[thm]{Lemma}
\theoremstyle{plain} \newtheorem{prop}[thm]{Proposition}
\theoremstyle{plain} \newtheorem{coro}[thm]{Corollary}
\theoremstyle{plain} 
\theoremstyle{plain} 
\theoremstyle{plain} \newtheorem{assu}[thm]{Assumption}
\theoremstyle{plain} \newtheorem*{claim}{Claim}
\theoremstyle{remark} 

\theoremstyle{remark} \newtheorem{rmk}[thm]{Remark}
\theoremstyle{remark} 
\theoremstyle{remark}

\usepackage{paralist}

\newcommand{\beq}{\begin{equation}}
	\newcommand{\eeq}{\end{equation}}
\newcommand{\bals}{\begin{align*}}  
	\newcommand{\bal}{\begin{align}}

		\newcommand{\cu}{\textbf}
		\newcommand{\real}{\mathbb{R}}
		\newcommand{\cpx}{\mathbb{C}}
		
		\newcommand{\posint}{\mathbb{N}}
		
		\def\a{\alpha}
		\def\b{\beta}
		\def\d{\delta}
		\def\e{\epsilon}
		
		\def\l{\lambda}

		\def\calA{\mathcal{A}}

		\usepackage{epsfig}
		
		\usepackage{multirow}
		\usepackage{subfigure}
		\usepackage{graphicx}
		\usepackage{float}
		\usepackage{booktabs}
		\usepackage{enumerate}
		
		\numberwithin{equation}{section}

		\usepackage{fancyhdr}
		\title{A priori estimates, uniqueness and non-degeneracy of positive solutions of the Choquard equation}
		
		\author{Zexing Li}
		\address{DPMMS\\University of Cambridge\\
			Cambridge\\ UK}
		\email{zl486@cam.ac.uk}
		
\begin{document}
			\maketitle
			
			\begin{abstract}
				We consider the positive solutions for the nonlocal Choquard equation $- \Delta u + u - (|\cdot|^{-\alpha} * |u|^p) |u|^{p-2} u = 0$ in $\real^d$. Compared with ground states, positive solutions form a larger class of solutions and lack variational information. Within the range of parameters of Ma-Zhao's result \cite{ma2010classification} on symmetry, we prove a priori estimates for positive solutions, generalizing the classical method of De Figueiredo-Lions-Nussbaum \cite{de1982priori} to the unbounded domain and the nonlocal nonlinearity in our model. As an application, we show uniqueness and non-degeneracy results for the positive solution of the Choquard equation when $d \in \{ 3, 4, 5\}$, $p \ge 2$ and $(\a, p)$ close to $(d-2, 2)$.
			\end{abstract}

			\tableofcontents
			
			\section{Introduction}
			
			\subsection{Introduction}
			
			In this paper, we consider the equation
			\beq
			- \Delta u + u - (|\cdot|^{-\alpha} * |u|^p) |u|^{p-2} u = 0\quad \text{in} \,\real^d \label{Choquard} \tag{Choquard}
			\eeq
			with $ d \ge 1,\,\alpha \in (0, d)$, $p \in (1, \infty)$ and $u$ a real-valued measurable function. 
			
			This equation (\ref{Choquard}) is usually referred to as Choquard or Choquard-Pekar equation. The case $d=3, \a = d-2, p = 2$ appears in various physical contexts, including quantum mechanics for polaron at rest \cite{pekar1954untersuchungen} and one-component plasma \cite{lieb1977existence}. It is also known as the Schr\"odinger-Newton equation by coupling the Schr\"odinger equation of quantum physics with nonrelativistic Newtonian gravity \cite{bahrami2014schrodinger}. Besides, every solution $u$ of (\ref{Choquard}) relates to a solitary wave solution $\psi(t, x) = e^{it} u(x)$ of the focusing time-dependent generalized Hartree equation
%			\footnote{Here $\psi$ is a complex-valued function. Thus actually the corresponding equation for stationary wave will be (\ref{Choquard}) with complex-valued $u$. But this paper focuses on positive solutions, so we consider (\ref{Choquard}) as an equation for real-valued functions.{\color{red} Not make sense (?) since \eqref{Choquard} is able to separate the real and imaginary parts.}
			\beq i\partial_t \psi = -\Delta \psi - (|\cdot|^{-\a} * |\psi|^p)|\psi|^{p-2} \psi,\quad \mathrm{in}\,\real_+ \times \real^d. \label{Hartree} \eeq
			When $p=2$, (\ref{Hartree}) is called Hartree equation, appearing in the study of Boson stars and other physical phenomena \cite{pitaevskii1961vortex}. Please refer to \cite{moroz2017guide} and the references therein for more mathematical and physics background of the Choquard equation (\ref{Choquard}).
			
			Solutions of (\ref{Choquard}) are formally critical points of the action functional
			\beq \mathcal{A}(u) := \frac{1}{2}\int_{\real^d} (|\nabla u |^2 + |u|^2) - \frac{1}{2p} \int_{\real^d} (|\cdot|^{-\a} * |u|^p) |u|^p. \eeq
			One of the most interesting solution is the groundstate $u$, defined as the minimizer of $\calA$ on the Nehari manifold
			\[\calA(u) = \inf\left\{ \calA(v): v\in H^1(\real^d) \backslash \{0\}, \langle \calA'(u), u\rangle = 0  \right\}. \]
			There are many studies of groundstates in the variational and elliptic viewpoint  \cite{LIONS19801063,lieb1977existence,moroz2013groundstates,moroz2017guide} and for the corresponding solitary wave in generalized Hartree equation (\ref{Hartree}) \cite{cazenave1982orbital,miao2009global} as well. 
			
			In this paper, we focus on a larger class of solutions for (\ref{Choquard}): the positive solutions. We say $u$ is a solution of (\ref{Choquard}) in the sense that $u \in H^1(\real^d) \cap L^{\frac{2dp}{2d-\a}}(\real^d)$ and for any test function $\varphi \in C^\infty_c(\real^d)$, 
			\[ \int_{\real^d} \left(\nabla u \cdot \nabla \varphi + u\varphi - (|\cdot|^{-\a} * |u|^p) |u|^{p-2} u \varphi \right) dy = 0.
			\]
			Note that groundstates must be positive from the variational structure and regularity properties (see \cite[Proposition 5.1]{moroz2013groundstates}). The lack of variational information makes the study of positive solutions rather harder. 
			
			Before coming to our results, we recall some basic results of positive solutions for (\ref{Choquard}). 
			\begin{thm}[{see \cite{moroz2013groundstates}}]\label{thmprop}
				For $d \ \ge 1$, $\a \in (0, d)$, $p \in (1, \infty)$ and 
				\beq \frac{d}{2d-\a} > \frac{1}{p} > \frac{d-2}{2d-\a}, \label{para3} \eeq
				then the following results hold.
				\begin{enumerate}
					\item (Existence) There exists at least one groundstate for (\ref{Choquard}). In particular, there exists at least one positive solution for (\ref{Choquard}). 
					\item (Regularity)  If $u$ is a positive solution for (\ref{Choquard}), then $u \in W^{2,r}(\real^d) \cap C^\infty_{loc}(\real^d)$ for $r \in (1, \infty)$. 
					\item (Decay) If $u$ is a positive solution for (\ref{Choquard}) and moreover $p \ge 2$, then there exists $\gamma > 0$ such that $u(x) \le C(u) e^{-\gamma|x|}$.
				\end{enumerate}
			\end{thm}
			
			To make use of all these properties, we will restrict our discussion to (\ref{Choquard}) with parameters $(d, \a, p)$ in the range
			\beq \begin{split}
				d \ge 1,\,\,\alpha \in (0, d),\, \,p \in (1, \infty),\quad
				\frac{1}{2} \ge \frac{1}{p} > \frac{d-2}{2d-\a}.
			\end{split}  \label{para2} \eeq 
			Notice that $p\ge 2$ also guarantee $\calA$ to be twice Fr\'echet-differentiable  on $H^1(\real^d)$ \cite[Proposition 3.1]{moroz2017guide}, which is essential in our discussion on uniqueness and non-degeneracy in \S \ref{s13}. 
			
			Next we recall a more involved result on the symmetry of positive solutions. For groundstates, the minimizing property enables a standard rearrangement argument, inferring the radially decreasing property around a fixed point\footnote{In this paper, a non-negative function $u$ is radially decreasing means $u(x) = u(|x|)$ and $\partial_r u(r) \le 0$. We say $u$ is radially decreasing around a fixed point $x_0$ to indicate $u(\cdot - x_0)$ is radially decreasing.} for parameters of full range (\ref{para3}) \cite[Proposition 5.2]{moroz2013groundstates}. As for positive solutions, we can merely rely on the information given by the elliptic equation. Ma-Zhao \cite{ma2010classification} managed to apply a moving plane method in the integral form to prove this symmetry in a narrower range. To clarify, we first define the following assumption on parameters:
			
			\begin{assu}\label{assu1}
				For $(d, \a, p)$ satisfying (\ref{para3}) and $p \ge 2$, we assume there exist constants
				\begin{align*}
					& r, r_1, r_2, r_3 \in \left[ 2, \frac{2d}{d-2} \right],\\
					& t, t_1 \in \left\{ t: \frac{1}{t} \in \left[ \frac{p(d-2)}{2d} - \frac{d-\a}{d}, \frac{\a}{d} \right)  \cap (0, 1) \right\}\\
					& \frac{1}{s} \in \left[\frac{2}{d} , 1 \right] \cap \left[ \frac{1}{r}, \frac{1}{r} + \frac{2}{d} \right]
				\end{align*}
				such that 
				$$ \left\{ \begin{array}{l}
					r_1 \ge p-2, r_2 \ge p-1, r_3 \ge p-1, \\
					\frac{1}{t_1}+ \frac{p-2}{r_1} +\frac{1}{r} = \frac{1}{s},\\
					\frac{p-1}{r_2} + \frac{1}{t} = \frac{1}{s}, \\
					\frac{1}{t}+\frac{d-\a}{d} =
					\frac{p-1}{r_3} + \frac{1}{r}.
				\end{array} \right.
				$$
			\end{assu}
			\begin{rmk}\label{rmkrange}
				This assumption still includes a wide range of interesting cases. \begin{itemize}
					\item One subcase (see \cite[Remark 3]{ma2010classification}) is that 
					\beq 2 < \a < d,\quad \frac{1}{2} \ge \frac{1}{p} > \frac{d-2}{2d-\a}. \label{para4}\eeq 
					In particular, if we define the critical scaling index
					$ s_c := \frac{d}{2} - \frac{d+2-\a}{2(p-1)}$
					for the generalized Hartree equation (\ref{Hartree})\footnote{That is, $\dot{H}^{s_c}(\real^d)$ norm of the solution to (\ref{Hartree}) is invariant under the scaling symmetry of (\ref{Hartree}).}, then (\ref{para4}) includes the whole interrange case $0 < s_c < 1$ when $p=2$.
					\item Another subcase is the perturbation of $(\a, p ) = (d-2, 2)$:
					\beq |\a - (d-2)| \le \frac{1}{100},\quad 0 \le p-2 \le \frac{1}{100}, \quad d \in \{ 3, 4, 5\}. \eeq
					The case $d = 5$ is contained in (\ref{para4}). For $d= 3, 4$ we can directly verify  
					\begin{align*}
						r = r_1 = \frac{8}{3}, r_2 = \frac{8}{3}(p-1), r_3 = \frac{12(p-1)}{3-4(\a -1)}, t = \frac{24}{7}, t_1 = \frac{24}{7 - 9(p-2)}, s = \frac{3}{2} \quad \mathrm{when}\,\, d=3;\\
						r = r_1 = \frac{8}{3}, r_2 = \frac{8(p-1)}{3}, r_3= \frac{8(p-1)}{3-2(\a - 2)}, t = 4, t_1 = \frac{8}{2-3(p-2)}, s = \frac{8}{5} \quad \mathrm{when}\,\, d=4.
					\end{align*}
				\end{itemize} 
			\end{rmk}

			Now we can state their result as 
			\begin{thm}[{\cite[Theorem 2]{ma2010classification}}]\label{mazhao}
				Any positive solution of (\ref{Choquard}) with $(d, \a, p)$ satisfying (\ref{para3}), $p \ge 2$ and Assumption \ref{assu1} must be radially decreasing around some fixed point. 
			\end{thm}

			This gives us motivation to study radially decreasing positive solutions of (\ref{Choquard}). 
			
			%			
			%			Indeed, (\ref{Choquard}) can be written as $\calA'(u) = 0$.The condition (\ref{para2}) ensures $\calA$
			%			to be twice Fr\'echet-differentiable ({\color{red} Fr\'echet-differentiable}) on $H^1(\real^d)$ and the existence of a groundstate solution $u \in H^1(\real^d)$ (see \cite[Proposition 3.1, Theorem 3.2]{moroz2017guide} and the reference therein). Here a groundstate refers to a minimizer of $\calA$ in $H^1(\real^d) \backslash\{0\}$, which is a positive solution of (\ref{Choquard}) \cite[Proposition 5.1]{moroz2013groundstates}.
			%			
			%			
			%			\begin{align}
			%				 d \ge 1,\quad \alpha \in (0, d)&, \quad p \in (1, \infty), \label{para1} \\
			%			 \frac{d-2}{2d-\a} < \frac{1}{p}& \le \frac{1}{2} \label{para2}.  
			%		    \end{align}
			
			\subsection{A priori estimates}\label{s12}
			
			Recall that Theorem \ref{thmprop} indicates that every positive solution of (\ref{Choquard}) is bounded in $W^{2, r}(\real^d)$ and decays exponentially. Our first main result indicates that if we further require the positive solution to be radially decreasing, it has an a priori upper bound, depending only on $(d, \a, p)$ in a uniform way.
			\begin{thm}\label{thmbound}
				For a radially decreasing positive solution $u$ of (\ref{Choquard}) with $(d, \a, p)$ satisfying (\ref{para2}), for $r \in (1, \infty)$, there exists constants $C(d, \a, p; r)$ and $C' (d, \a, p)$ such that 
				\begin{align}
					\| u \|_{W^{2, r}(\real^d)} &\le C(d, \a, p; r),\label{aprioribound} \\
					|\nabla u(s)| + u(s) &\le C'(d, \a, p) e^{ -\frac{s}{2}},\quad s \ge 0. \label{aprioribound2}
				\end{align}
				Moreover, these constants depend continuously on $(\a, p)$, indicating that this is a uniform bound for $(\alpha, p)$ satisfying (\ref{para2}) taking values in a compact subset.
			\end{thm}
			
			Combine Theorem \ref{mazhao} and Theorem \ref{thmbound}, we can remove the radially decreasing condition under Assumption \ref{assu1}.
			
			\begin{thm}
				For any positive solution $u$ of the Choquard equation (\ref{Choquard}) with $(d, \a, p)$ satisfying (\ref{para2}) and Assumption \ref{assu1}, we have the same uniform a priori bounds (\ref{aprioribound}) and (\ref{aprioribound2}).
			\end{thm}

			%		    \begin{rmk}\label{rmkbound}
			%		    	From the improvement of regularity proposition (\cite[Proposition 4.1]{moroz2013groundstates}, recorded in this paper as Proposition \ref{propreg}), any $H^1$ solution of (\ref{Choquard}) with parameters satisfying (\ref{para1}) and (\ref{para2}) belongs to $W^{2, r}$ for any $r \in (1, \infty)$. In fact, this argument is applied to reduce this theorem to a priori bound of $L^2$ norm (see Step 5 of \S 3.2 for details).
			%		    \end{rmk}

			There is a long history of studying a priori bounds for positive solutions of elliptic equations \cite{de1982priori,gidas1981priori, souto1995priori,de2001priori,ruiz2004priori,barrios2018priori}. Such a priori bounds provide lots of information about existence of positive solution and the structure of the positive solutions set, see \cite{brezis1977class,lions1982existence,ni1985uniqueness} and \S \ref{s13}. In more recent work of proving uniqueness for nonlinear groundstates of fractional Laplacians \cite{frank2016uniqueness,frank2013uniqueness}, one crucial step is to derive an a priori bound for the global bifurcation branch.
			%	       The authors control the groundstates via its "evolution" along the bifurcation branch, while in our result we prove a priori estimates in a more direct way in our model. 
			
			Among these works, three methods have been used to derive a priori bounds. In \cite{frank2013uniqueness}, Frank-Lenzmann controls the solutions in the branch through its "evolution" in the parameter space. It requires a non-degeneracy result to construct the bifurcation branch, which is highly non-trivial. The second approach is a blowup method from Gidas-Spruck \cite{gidas1981priori}. It is a contradictory argument, reducing the problem of a priori bounds to Liouville result of some simple equation by rescaling. This method is powerful and may also work for our problem to get an $L^\infty$ bound. However, this is not enough to control $L^p$ norm in unbounded space, and such contradictory argument provides relatively little information on the shape of the solution. The third method due to De Figueiredo-Lions-Nussbaum \cite{de1982priori} derives the bound in a direct way. They exploit the positive eigenfunction to get some a priori bound and get the desired bound with functional identities and the subcriticality nature of the equation. We adapt this method to our problem. 
			
			To our knowledge, however, the second and third method have only been applied for problems on bounded domains with local nonlinearities (although the blowup method can tackle systems of equations, see for example \cite{souto1995priori}). Perhaps the main innovation of this paper is to deal with the unboundedness of domain and nonlocality of nonlinearity. Nevertheless, one good point in our setting is the symmetry of $\real^d$, so we can add radially decreasing property to our positive solutions thanks to Theorem \ref{mazhao}.
			
			The unboundedness of domain causes trouble in two ways. On the one hand, we lack of eigenfunctions of Laplacian to get initial a priori information. This can be substituted by a nonlinear positive eigenfunction, namely the groundstate of nonlinear elliptic equation. It gives weaker information but enough for us. On the other hand, a more challenging problem is to control the solution in unbounded domains. Naturally, the exponential decay when $u$ small is a good enough bound in some exterior region. The local argument in \cite{de1982priori} basically ensures good bound for any fixed interior region, but the question is whether the non-exponential-decay interior region can be arbitrarily large. The crux is to exclude flatness in the connection region. We apply the a priori information and an ODE argument for this.
			%	       Thanks to the radial symmetry, we can estimate the size of the interior region by some detailed discussion of the ODE evolution with functional identities.
			As a model case, we study the following semilinear elliptic equation with a local nonlinearity
			\beq \label{semi} -\Delta u + u - |u|^{p-1} u = 0 \quad \text{in}\,\real^d \tag{Model} \eeq
			where $p \in (1, \infty)$ for $d = 1, 2$ and $p \in (1, \frac{d+2}{d-2})$ for $d \ge 3$. We have the following conclusion
			\begin{thm}\label{thmsemi}
				The positive $H^1$ solution $u$ of (\ref{semi}) has a priori bounds
				\beq \| u \|_{H^1} \le C(p, d), \eeq
				where $C(p, d)$ depends continuously on $p$. 
			\end{thm}
			Although this result can be obtained from the existence \cite{berestycki1979existence} and uniqueness \cite{kwong1989uniqueness} of the positive solution, this theorem may still of its own interest and possible to generalize to more complicated nonlinearity. 
			
			Regarding the second difficulty of (\ref{Choquard}), the nonlocal nonlinearity $(|\cdot|^{-\a} * |u|^p) |u|^{p-2} u$, 
			we will exploit the good struture of this nonlocality to bound some norm of $u$ (see Proposition \ref{propnonlocal1}). Also, we view it as another radial function and study its evolution (Proposition \ref{propnonlocalevo}) along with that of $u$. Finally, after generalizing the argument of subcriticality in \cite{de1982priori} to nonlocal case, we can complete the proof of Theorem \ref{thmbound} in a similar framework as Theorem \ref{thmsemi}.
			
			\subsection{Non-degeneracy and uniqueness}\label{s13}
			
			Thereafter, we consider the uniqueness and non-degeneracy (explained later) of the positive solution of (\ref{Choquard}). 
			These results are essential for discussing the dynamics of the corresponding solitary wave solution $\psi(t, x) = e^{it} u(x)$ of the focusing time-dependent generalized Hartree equation. The uniqueness clarifies what makes up the minimal obstruction of global-wellposedness and scattering, and the non-degeneracy provides suitable spectral condition for perturbative analysis. See \cite{zhou2022threshold,krieger2009two,krieger2009stability} and see \cite{lenzmann2009uniqueness} for more applications. 
			
			However, relatively little is known on the uniqueness and non-degeneracy for positive solutions or groundstates of (\ref{Choquard}). For the isolated case $(d, \a, p) = (3, 1, 2)$, Lieb \cite{lieb1977existence} used the special structure of Newtonian potential $|x|^{-(d-2)}$ to prove the uniqueness of the radial positive solution (hence also of the groundstate) and non-degeneracy was verified by \cite{tod1999analytical,wei2009strongly,lenzmann2009uniqueness}. These results can be easily generalized to the positive solution of $d \in \{4, 5\}$, $(\a ,p) = (d-2, 2)$ ( see \cite[Appendix A]{arora2019global} for uniqueness and \cite{chen2021nondegeneracy} for non-degeneracy). The radial condition was removed by Ma-Zhao's result \cite{ma2010classification}. Besides, we have uniqueness and non-degeneracy of the groundstate for $d \ge 3$, $p \in (2, \frac{2d}{d-2})$ and $\a$ close to $0$ or $d$ \cite{seok2019limit} and for $d\in \{ 3, 4, 5\}$, $\a = d-2$ and $p \in [2, 2+\d]$ for some $\d > 0$ \cite{xiang2016uniqueness}. Both are proved by perturbative arguments. We also mention another nonlocal elliptic problem, the nonlinear fractional Laplacian equation, where the uniqueness and non-degeneracy of groundstates were resolved partially in \cite{fall2014uniqueness} and later completely in \cite{frank2016uniqueness,frank2013uniqueness}. All these results for continuous exponents utilize more or less variational information of groundstates and thus not work for positive solutions. 
			
			In this paper, with our a priori bound Theorem \ref{thmbound}, we prove the uniqueness and nondegeneracy of the positive solution for (\ref{Choquard}) with parameters $d \in \{3, 4, 5\}$ and $(\a, p) \in [d-2-\d, d-2+\d] \times [2, 2+\d]$ for some $\d > 0$. It can be viewed as generalization of \cite{xiang2016uniqueness} to all positive solutions, two-dimensional perturbation of parameters and also to multiple dimensions, or generalization of \cite{ma2010classification} to a neighborhood of exponents.
			
			Define the corresponding linearized operator $L_{+}$  associated with a function $Q$ and for parameters $(d, \a, p)$ to be\footnote{Usually, we take $Q$ to be $Q_{d, \a, p}$, a positive solution for (\ref{Choquard}) with these parameters. So we may omit these parameters and denote $L_{+, Q_{d, \a, p}, d, \a, p}$ as $L_{+, Q_{d, \a, p}}$ for simplicity. We may even leave out $d$ if $(\a, p)$ are emphasized and no ambiguity occurs.}
			\beq
			L_{+, Q, d, \a, p} \xi := -\Delta \xi + \xi - (p-1) (|\cdot|^{-\alpha} * Q^p) Q^{p-2} \xi - p (|\cdot|^{-\alpha} * (Q^{p-1} \xi) )Q^{p-1} \label{Lplus}  \eeq
			as a nonlocal operator on $L^2(\real^d)$. Then we can state our results on non-degeneracy and uniqueness
			\begin{thm}[Non-degeneracy]\label{thmnondeg}
				For $d \in \{3, 4, 5\}$, There exists $\delta > 0$ such that for   $(\a, p) \in [d-2 - \delta, d-2+\d] \times [2, 2+\d]$, any positive solution $Q_{\a, p}$ of (\ref{Choquard}) is non-degenerate, namely
				\beq Ker L_{+, Q_{ \a, p}} = \mathrm{span}\{\partial_i Q_{\a, p}\}_{i=1}^d \label{nondeg} \eeq
			\end{thm}
			\begin{thm}[Uniqueness]\label{thmuni}
				For $d \in \{3, 4, 5\}$, There exists $\delta > 0$ such that for   $(\a, p) \in [d-2 - \delta, d-2+\d] \times [2, 2+\d]$, the positive solution  of (\ref{Choquard}) is unique up to translations.
			\end{thm}
			
			We remark that $Ker L_{+, Q_{ \a, p}} \supset \mathrm{span}\{\partial_i Q_{\a, p}\}_{i=1}^d$ holds for any solution $Q_{\a, p}$ of (\ref{Choquard}) by differentiating the equation. So (\ref{nondeg}) indicates that there exist no other vanishing modes for $L_+$, which explains the meaning of non-degeneracy. For more illustration of non-degeneracy and the operator $L_+$, please see \cite{lenzmann2009uniqueness,xiang2016uniqueness}.
			
			As for the proof, our starting point is the uniqueness and non-degeneracy for $d \in \{3, 4, 5 \}$, $(\a, p) = (d-2, 2)$ (see \cite[Appendix A]{arora2019global} and \cite{chen2021nondegeneracy}). 
			Theorem \ref{mazhao} and Remark \ref{rmkrange} reduce these theorems our questions to the discussion of radially decreasing positive solutions (with $\d \le \frac{1}{100}$), and then we use a perturbative strategy. The first ingredient is a compactness result Proposition \ref{propasym}, confirming that every positive solution approximates $Q_{d-2, 2}$ (the unique solution for $(\a, p) = (d-2, 2)$) when the parameter $(\a, p)$ approximates. It is here that we exploit the a priori bounds Theorem \ref{thmbound} to discuss such asymptotic behavior for positive solutions rather than groundstates. Then since $0$ is an isolated eigenvalue for $L_{+,Q_{d-2, 2}}$, the non-degeneracy Theorem \ref{thmnondeg} comes from perturbation of Fredholm operator. The uniqueness Theorem \ref{thmuni} easily follows this compactness result and a local uniqueness theorem (Proposition \ref{proplocuni}) from implicit function theorem. We remark that the linearized operator is actually not $C^1$ \emph{near} $Q_{d-2, 2}$ but just \emph{at} $Q_{d-2,2}$ (see Lemma \ref{lemregf}), which requires us to be more careful when applying the implicit function theorem to prove Proposition \ref{proplocuni}. 
			
			\subsection{Structure of the paper and notations} The structure of this paper is as follows. In \S \ref{s2}, we show some estimates of nonlinearity and functional identity for the Choquard equation (\ref{Choquard}) as a preparation. Then we prove the a priori bounds for the model case (Theorem \ref{thmsemi}) and Choquard equation (\ref{thmbound}) in \S \ref{s31} and \S \ref{s32} respectively. The non-degeneracy and uniqueness (Theorem \ref{thmnondeg} and \ref{thmuni}) follow in \S \ref{s4}. We put some complicated computations and the refined implicit function theorem in the appendix, which are used in the last two sections. 
			
			Our notations are standard. We employ $L^p(\real^d)$, $W^{k, r} (\real^d)$, $H^k(\real^d)$ and $C^k(\real^d)$ for those Sobolev spaces of real-valued functions.  For Banach space $X$ and $Y$, we use $L(X, Y)$ to denote the space of bounded linear operators from $X$ to $Y$. In particular, $L(X) := L(X, X)$. We use $B_R(x)$ with $x \in \real^d$ to denote the Euclidean open ball centered at $x$ with radius $R > 0$. Two related notations are $B_R^c(x) := \real^d \backslash B_R(x)$ and $B_R := B_R(0)$. 
			
			We also write $X\lesssim Y$ or $Y\gtrsim X$ to indicate $X\leq CY$ for
			some constant $C>0$, and write $X \sim Y$ for $X \lesssim Y \lesssim X$. If $C$ depends upon some additional parameters, we will indicate this with subscripts. For example, $X\lesssim_\phi Y$ means $X\leq C(\phi) Y$.
			We use $X_n = o_{n}(Y_n)$ to denote that for any $\e > 0$, there exists $N > 0$ such that $|X_n| \le \e Y_n$ for any $n \ge N$. The subscript can also other parameters with prescribed limiting process, for example $p \to 2$.

			\section{Preliminaries for the Choquard equation}\label{s2}

			%		We start with an elementary observation without proof. 
			%		\begin{lem}
			%		    Let $f$ be a nonnegative radial function in $\real^d$. For 
			%		\end{lem}

			In this section, we first show some propositions for the nonlocal term $|\cdot|^{-\a} * f$ (also called Riesz potential) when $f$ is non-negative and radially decreasing. Evidently, $|\cdot|^{-\a} * f$ is also a non-negative and radial function.
			
			We recall a pointwise equivalence result, which is essential in our proof of a priori bounds. 
			
			\begin{prop}[{\cite[corollary 2.3]{duoandikoetxea2013fractional}}] Let $f$ be a non-negative, radially decreasing function in $\real^d$ and $\a \in (0, d)$. We have
				\beq \left(|\cdot|^{-\a} * f\right)(r) \sim_{d, \a} r^{-\a} \int_0^r f(s) s^{d-1} ds + \int_r^\infty f(s) s^{d-1-\a}ds  \label{nonlocal1} \eeq
				Moreover, the constant is uniformly bounded for $\a$ in a compact subset of $(0, d)$. \label{propnonlocal1}
			\end{prop}
			We recall its proof for completeness.
			\begin{proof}
				Firsly we claim that for any $r \ge 0$ and $x\in \real^d$ with $|x| = r$, 
				\beq \left(|\cdot|^{-\a} * f\right)(r) \sim_{d, \a} r^{-\a} \int_{B_r} f(y) dy + \int_{B_r^{c}} f(y)|y|^{-\a} dy + \int_{B_\frac{r}{2}(x)} \frac{f(y)}{|x-y|^{-\a}} dy.\label{non11} \eeq
				Those three terms on the right hand side correspond respectively to integration on $B_r \backslash B_\frac{r}{2}(x)$, $B_r^{c}\backslash B_\frac{r}{2}(x)$ and $B_\frac{r}{2}(x)$. The equivalence requires non-negativity.  
				
				Next we will prove 
				\beq \int_{B_\frac{r}{2}(x)} \frac{f(y)}{|x-y|^{-\a}} \lesssim_{d, \a} r^{-\a} \int_{0}^r f(s) s^{d-1} ds. \label{non12}\eeq
				which combined with (\ref{non11}) implies (\ref{nonlocal1}). The radially decreasing property implies
				\[ f\left(\frac{r}{2}\right) \le \frac{4}{r}\int_{\frac{r}{4}}^\frac{r}{2} f(s)\left(\frac{4s}{r}\right)^{d-1} ds \le \frac{4^d}{r^d} \int_0^r f(s)s^{d-1}ds \] 
				and 
				\[ \int_{B_\frac{r}{2}(x)} \frac{f(y)}{|x-y|^{-\a}} \lesssim_d  f\left(\frac{r}{2}\right) \left(\frac{r}{2}\right)^{d-\a}. \]
				They yield (\ref{non12}).
			\end{proof}
			
			Then we discuss the evolution of $|\cdot|^{-\a} * f$.
			
			\begin{prop}\label{propnonlocalevo} Let $f$ be a non-negative, radially decreasing $C^1$ function in $\real^d$ and $\a \in (0, d)$. Then for any $\epsilon_0 \le \frac{1}{4}$, we have 
				\beq 
				-\partial_r\left(|\cdot|^{-\a} * f\right)(r) \gtrsim_{d, \a, \epsilon_0}\left[ f\left(\frac{2}{3}r\right) - f\left(\frac{1}{2\epsilon_0}r\right) \right] r^{d-1-\a}. \label{nonlocalevo} \eeq
				The constant here is uniformly bounded for $(\a, \e_0)$ taking values in compact subset of $(0, d) \times (0, \frac{1}{4}]$. 
				In particular, $|\cdot|^{-\a} * f$ is strictly radially decreasing if $f$ vanishes at infinity.
			\end{prop}
			
			As a preparation, we define $\chi_\Omega$ to be the characteristic function of $\Omega \subset \real^d$ and
			\beq \Psi_{R_1, R_2} (x):= (\chi_{B_{R_1}} * \chi_{B_{R_2}}) (x) = \left| B_{R_1}(0) \cap B_{R_2}(x)  \right|.\label{psi1} \eeq
			Some properties of $\Psi$ are discussed in the lemma below.
			
			\begin{lem}[Properties of $\Psi$] For $R_1 \ge R_2 > 0$, $\Psi_{R_1, R_2}$ satisfies the followings:
				\begin{enumerate}[(1)]
					\item $\Psi_{R_1,R_2}$ is non-negative and radially decreasing. $\Psi_{R_1, R_2} = \Psi_{R_2, R_1}$. $\Psi_{R_1, R_2}(x) = R_2^d$ in $B_{R_1 - R_2}$ and $\Psi_{R_1, R_2}(x) = 0$ in $B_{R_1 + R_2}^c$.
					\item Scaling property: 
					\[\Psi_{R_1, R_2}(r) = R_2^{d} \Psi_{\frac{R_1}{R_2}, 1}\left( \frac{r}{R_2}\right)\]
					\item Derivative: As a radial function, $\Psi_{R_1, R_2}$ is absolutely continuous on $[0, \infty)$ and
					\begin{align}
						-\frac{d}{dr} \Psi_{R_1, R_2}(r) = \left|\partial B_{R_1}(0) \cap B_{R_2}(r)  \right|,\quad r \in (R_1 -R_2, R_1 + R_2). \end{align}
					\item Lower bound of the derivative: for $\epsilon_0 \in (0,1)$,
					\begin{align}
						-\frac{d}{dr} \Psi_{R_1, R_2}(r) &\gtrsim_{d, \e_0} R_2^{d-1},\quad \text{for}\,\,\, |r-R_1| \le (1-\e_0) R_2. \label{lopsi}
					\end{align}
				\end{enumerate}
			\end{lem}
			
			\begin{proof}
				(1)-(3) directly follows the definition (\ref{psi1}). From (2) and (3), the estimate (4) boils down to
				\beq \inf_{R\ge 1} \inf_{t \in [-1+\e_0, 1-\e_0]} \left| \partial B_R(0) \cap B_1 (R+t) \right| > 0 \label{psilo} \eeq
				for any $\epsilon_0 \in (0, 1)$. For $d = 1$, $\mathrm{LHS\,\,of\,\,}(\ref{psilo}) = 1 > 0$. Next we consider the case $d \ge 2$.
				
				Assuming the center of $B_1(R+t)$ to be $(R+t, 0,\ldots,0) \in \real^d$, the symmetry around $x_1$-axis implies that $\partial B_R(0) \cap \partial B_1(R+t)$ lies in a hyperplane $\{ x_1 = a(R, t) \}$ with $a(R, t) \in (0, R)$. And 
				\[\partial B_R(0) \cap B_1 (R+t) = \partial B_R(0) \cap\{ x_1 \ge a(R, t)\}. \]
				Using spherical coordinates, we see for $d \ge 3$ (and similar for $d = 2$)
				\begin{align*}
					|\partial B_R(0) \cap\{ x_1 \ge a(R, t)\} | &= R^{d-1} \int_0^{\arccos \frac{a}{R}} \int_0^\pi \cdots \int_0^{2\pi} \left(\prod_{k=1}^{d-2} \sin^k \theta_{d-1-k}\right) d\theta_{d-1} d\theta_{d-2} \cdots d\theta_1 \\
					\sim_d R^{d-1} &\int_0^{\arccos \frac{a}{R}} \sin^{d-2} \theta_1 d\theta_1 \sim_d R^{d-1} \left(\frac{R-a}{R}\right)^{\frac{d-1}{2}} \sim_d \left( R^2 - a^2 \right)^{\frac{d-1}{2}},
				\end{align*}
			    where we made use of 
			    \[ \frac{R-a}{R} = \cos 0 - \cos \arccos \frac{a}{R} = \int_0^{\arccos \frac a R}sin\theta d \theta \sim \left(\arccos \frac a R\right)^2. \]
				Thus it suffices to give a uniform lower bound on $(R^2-a^2(R, t))^{\frac{1}{2}}$, which is the radius of the $(d-1)$-dimensional ball $\partial B_R(0) \cap \{x_1 = a(R, t) \}$. Actually, we are just using the measure of this sectional ball to bound the measure of the corresponding spherical cap. Elementary trigonometric calculations and estimates indicate that
				\begin{align*}
					R^2 - a(R, t)^2 = \frac{R}{R+t}\left(1 - \frac{1-t^2}{4R(R+t)}\right) (1-t^2) \ge \frac{1}{2}\left(1 - \frac{1-t}{4R} \right) (1-t^2) \ge \frac{\e_0}{4}
				\end{align*}
				for any $R \ge 1$ and $|t| \le 1-\e_0$. 
			\end{proof}
			
			\begin{proof}[Proof of Proposition \ref{propnonlocalevo}]
				For the given $f$, denote $r_{\l} := \inf\{ r \ge 0: f(r) \le \l\}$.  
				Using the layer cake representation, we have 
				\begin{align*}
					(|\cdot |^{-\a} * f)(x) &= \int_0^\infty \int_0^\infty \int_{\real^d} \chi_{\{z:|z|^{-\a} >t\}}(y) \chi_{\{z: f(z) > s\}}(x-y) dy ds dt \\
					&= \int_0^\infty \int_0^\infty \Psi_{r_s, t^{-\frac{1}{\a}}}(x) ds dt  
				\end{align*}
				Thus from (\ref{lopsi}), we have
				\begin{align*}
					-\frac{d}{dr} (|\cdot |^{-\a} * f)(r)& \gtrsim_{d, \e_0} \int_0^\infty \int_{r_s^{-\a}}^\infty \chi_{\{r: |r - r_s| \le (1-\e_0) t^{-\frac{1}{\a}}\}}(r) t^{-\frac{d-1}{\a}} dt ds \\
					&= \int_{f\left(\frac{r}{\e_0}\right)}^{f\left(\frac{r}{2-\e_0}\right)} \int_{r_s^{-\a}}^{\left(\frac{|r-r_s|}{1-\e_0}\right)^{-\a}} t^{-\frac{d-1}{\a}} dt ds 
				\end{align*}
				The last equality comes from this observation: the inner integral is non-zero only if $|r-r_s| \le (1-\e_0)r_s $, which means $\frac{r}{2-\e_0} \le r_s \le \frac{r}{\e_0}$ and thus $ f\left(\frac{r}{2-\e_0}\right) \le s \le  f\left(\frac{r}{\e_0}\right)$ from the definition of $r_s$. By changing variables $\tilde{t} := r^\a t$, the inner integral turns into
				\begin{align}
					\int_{r_s^{-\a}}^{\left(\frac{|r-r_s|}{1-\e_0}\right)^{-\a}} t^{-\frac{d-1}{\a}} dt 
					= r^{d-1-\a} \int_{\left( \frac{r_s}{r}\right)^{-\a}}^{\left(\frac{|1-\left( \frac{r_s}{r}\right)|}{1-\e_0}\right)^{-\a}} \tilde{t}^{-\frac{d-1}{\a}} d\tilde{t} 
					%    	&= r^{d-1-\a}\left\{ \begin{array}{ll}
					%    	\frac{\a}{d-1-\a} \left[\left(\frac{|1-\left( \frac{r_s}{r}\right) |}{1-\e_0}\right)^{d-1-\a} - \left( \frac{r_s}{r}\right) ^{d-1-\a} \right],& \a \in (0, d) \backslash \{d-1\} \\
					%        (d-1)\ln \left(\frac{|\left( \frac{r}{r_s}\right) - 1 |}{1-\e_0}\right)   , & \a = d-1  \end{array}\right. \\
					=: r^{d-1-\a}g_{\e_0, \a} \left( \frac{r_s}{r}\right)  \label{gtdef}
				\end{align}
				Now we claim 
				\beq g_{\e_0, \a} (t) \gtrsim_{d, \a, \e_0} 1 \quad t \in \left[\frac{2}{3}, \frac{1}{2\epsilon_0} \right]. \label{gtest}  \eeq 
				Since $\frac{r_s}{r} \in \left[\frac{2}{3}, \frac{1}{2\epsilon_0}\right]$ is equivalent to $s \in \left[ f\left(\frac{2r}{3}\right),  f\left(\frac{r}{2\e_0}\right) \right]$ and $\e_0 \le \frac{1}{4}$ indicates $\frac{1}{2-\e_0} < \frac{2}{3}$, 
				\begin{align*} -\frac{d}{dr} (|\cdot |^{-\a} * f)(r) \gtrsim_{d, \e_0} \int^{f\left(\frac{2r}{3}\right)}_{f\left(\frac{r}{2\e_0}\right)} r^{d-1-\a}g_{\e_0, \a} \left( \frac{r_s}{r}\right) ds
					\gtrsim_{d,\a,\e_0} \left[ f\left(\frac{2}{3}r\right) - f\left(\frac{1}{2\epsilon_0}r\right) \right] r^{d-1-\a}.
				\end{align*}
				
				Finally we check the claim (\ref{gtest}). Computing the integral in (\ref{gtdef}), we also get another representation of $g_{\e_0, \a}$
				\[ g_{\e_0, \a}(t) = \left\{ \begin{array}{ll}
					\frac{\a}{d-1-\a} \left[t^{d-1-\a} -\left(\frac{|1-t|}{1-\e_0}\right)^{d-1-\a} \right],& \a \in (0, d) \backslash \{d-1\}, \\
					-(d-1)\ln \left(\frac{|t^{-1} - 1 |}{1-\e_0}\right)   , & \a = d-1.  \end{array}\right.\]
				Evidently, $g_{\e_0, \a}(t)$ is $C^1$ on $[\frac{1}{2-\e_0}, 1) \cup (1, \frac{1}{2\e_0}]$. Also at $t = 1$, we always have $\lim_{t\rightarrow 1-0}g_{\e_0, \a} (t) = \lim_{t\rightarrow 1+0}g_{\e_0, \a} (t)$. Thus an elementary computation on $g_{\e_0, \a}'(t)$ verifies that for all $\a \in (0, d)$ and $\e_0 \in (0, 1/4]$, $g_{\e_0, \a}$ first monotonically increases and then monotonically decreases with respect to $t \in [\frac{1}{2-\e_0}, \frac{1}{\e_0}]$. In particular,
				\beq \inf_{t \in \left[\frac{2}{3}, \frac{1}{2\epsilon_0} \right]} g_{\e_0, \a} (t) = \min \left\{ g_{\e_0, \a}\left(\frac{2}{3}\right),  g_{\e_0, \a}\left(\frac{1}{2{\e_0}}\right)\right\} \gtrsim_{d, \a, \e_0} 1.\nonumber \eeq
				The last inequality comes easily from the integral representation (\ref{gtdef}). And that complete the proof of (\ref{gtest}) and this proposition. 
			\end{proof}
			
			Finally, we recall some useful functional identities for solutions of Choquard equations. 
			
			%    We first state the improvement of regularity, which enables us to discuss a priori bound of $W^{2,r}$ norm (see Remark \ref{rmkbound}) and to derive pointwise estimates of $u(r)$ and $-\partial_r u(r)$.
			%    
			%    \begin{prop}[Regulariy, {\cite[Proposition 4.1]{moroz2013groundstates}}] \label{propreg}
			%    	Let $u$ be an $H^1$ solution of (\ref{Choquard}) with parameters $(d, \a, p)$ satisfying (\ref{para1}) and (\ref{para2}). Then $u \in W^{2, r}(\real^d)$ for any $r \in (1, \infty)$. Moreover, if $u$ is positive, then $u \in C^\infty_{loc}(\real^d)$.
			%    \end{prop}
			
			\begin{prop}[Functional Identities, {\cite[Proposition 3.1]{moroz2013groundstates}}]\label{propfuncid}
				Let $u\in H^1(\real^d)$ be a solution of (\ref{Choquard}) with parameters $(d, \a, p)$ satisfying (\ref{para2}). Then
				\begin{align}
					\| \nabla u\|_{L^2}^2 + \| u \|_{L^2}^2 &= \int_{\real^d} \left(|\cdot|^{-\a} * |u|^p\right) |u|^{p} dx \label{func01} \\
					\frac{d-2}{2}\| \nabla u\|_{L^2}^2 + \frac{d}{2}\| u \|_{L^2}^2 &= \frac{2d-\a}{2p} \int_{\real^d} \left(|\cdot|^{-\a} * |u|^p\right) |u|^{p} dx \label{func02}
				\end{align}
				In particular, 
				\begin{align} \| \nabla u\|_{L^2}^2 &= \frac{pd - (2d-\a)}{2p} \int_{\real^d} \left(|\cdot|^{-\a} * |u|^p\right) |u|^{p} dx  \label{func03} \\
					\|  u\|_{L^2}^2 &= \frac{(2d-\a) - p(d-2)}{2p} \int_{\real^d} \left(|\cdot|^{-\a} * |u|^p\right) |u|^{p} dx  \label{func04}
				\end{align}
			\end{prop}
			
			The last two inequalities follow directly from the first two, which are derived by taking inner product of (\ref{Choquard}) with $u$ and $x \cdot \nabla u$ respectively. Theorem \ref{thmprop}(2) ensures enough smoothness to do integration by parts. We mention that (\ref{func02}) is usually called Pohozaev identity. During the proof of Theorem \ref{thmbound}, we will also derive and utilize a localized version of them ((\ref{func1}) and (\ref{func2})). 
			
			An immediate corollary is the following lower bound of $H^1$ norm.
			\begin{coro}\label{corolo}
				Let $u\in H^1(\real^d)$ be a solution of (\ref{Choquard}) with parameters $(d, \a, p)$ satisfying (\ref{para2}). In addition, suppose $u$ is not zero. Then 
				\beq \| u \|_{H^1(\real^d)} \gtrsim_{d, \a, p} 1  \eeq
				where the constant depends continuously on $(\a, p)$.
			\end{coro}
			\begin{proof}
				By H\"older inequality, Hardy-Littlewood-Sobolev inequality and Sobolev embedding, whose constants depend on the index continuously, we see
				\[ \int_{\real^d}\left(|\cdot|^{-\a} * |u|^p\right) |u|^{p} dx \le \left\| |\cdot|^{-\a} * |u|^p \right\|_{L^\frac{2d}{\a}} \| |u|^p\|_{L^{\frac{2d-\a}{\a}}}\lesssim_{d, \a, p} \|u \|_{L^{\frac{2dp}{2d-\a}}}^{2p} \lesssim_{d, \a, p} \| u \|_{H^1}^{2p} \]
				where we used (\ref{para2}) to check that $H^1 \hookrightarrow L^{\frac{2dp}{2d-\a}}$. Also from (\ref{func01}) and $p > 1$, we have
				\[\| u \|_{H^1}^2 \lesssim_{d, \a, p} \| u \|_{H^1}^{2p}, \]
				which confirms the lower bound if $u$ is not identically zero.
			\end{proof}
			
			\section{A priori estimates for the model case} \label{s31}
			
			In this subsection, we prove the a priori bound Theorem \ref{thmsemi} for (\ref{semi}) with $p \in (1, \infty)$ if $d = 1, 2$, $p \in (1, \frac{d+2}{d-2})$ if $d \ge 3$.
			
			By moving plane method \cite{gidas1981symmetry}, one can show that any $H^1$ positive solution of (\ref{semi}) will be radially decreasing around some fixed point. So again we only need to prove that bound for positive, radially decreasing solution $u$. Also we know $u$ is Schwartz from a standard iterative argument to improve the regularity (see for example \cite[ Proposition B.7]{tao2006nonlinear}), which enables us to freely taking derivatives and integrating by parts.
			
			\cu{Step 1. A priori nonlinear eigenfunction estimate}
			
			We fix $d$ and arbitrarily take a $p_0 = p_0(d)$ within the range above. Then we take a Schwartz, positive, radially decreasing solution\footnote{For example, we may take a minimizer of the corresponding Weinstein functional (see \cite{weinstein1983nonlinear} or  \cite[Appendix B]{tao2006nonlinear}). } $Q$ of (\ref{semi}). That is,
			\[-\Delta Q + Q - Q^{p_0} = 0 \quad \mathrm{in}\,\,\real^d. \] 
			Using the equation of $Q$ and $u$, we have
			\begin{align}\label{intgQ}
				\int u^p Q = \int (-\Delta + 1)u Q = \int u (-\Delta+1) Q = \int u Q^{p_0}
			\end{align}
			Fix 
			$C_0 := 2\| Q \|_{L^\infty}^{p_0 - 1}$,
			we have 
			\begin{align*}
				C_0 \int u Q &= \int_{\{x: u(x) \le C_0^{\frac{1}{p - 1} }\} } C_0 u Q + \int_{\{ x:u(x) > C_0^{\frac{1}{p - 1}}\} } C_0 u Q \\
				&\le  \int u^p Q +C_0^{\frac{p}{p-1}} \int Q 
				=  \int u Q^{p_0} + C_0^{\frac{p}{p-1}} \int Q \\
				&\le  \frac{C_0}{2} \int u Q + C_0^{\frac{p}{p-1}} \int Q.
			\end{align*}
			
			Thus 
			\beq  \int u Q \le 2C_0^{\frac{1}{p-1}}\| Q\|_{L^1}\lesssim_{d, p} 1.   \label{eigenbound1}\eeq
			With (\ref{intgQ}), we also obtain
			\beq \int u^p Q \le C_0^{\frac{p}{p-1}}\| Q\|_{L^1} \lesssim_{d, p} 1. \label{eigenbound2}\eeq
			
			\cu{Step 2. Exponential decay far away}
			
			Since $u$ is radially decreasing and positive, we denote $r_\l$ with $\lambda \in (0, u(0))$ to be
			\beq r_\lambda := \inf \{r \ge 0: u(r) \le \l \}. \eeq
			Consider 
			\beq \d_0 = (2p)^{-\frac{1}{p-1}} \sim_p 1, \label{d0}\eeq 
			and 
			\beq R_0 := \max\{r_{\d_0} , d^{\frac{1}{2}}\}. \label{R} \eeq
			Then for any $r \ge R \ge R_0$, we will characterize the exponential decay of $u(r)$ and $-\partial_r u(r)$ related to $R$.
			
			For $r \ge r_{\d_0}$, $u(r)^{p-1} \le (2p)^{-1}$ and hence 
			\[ \left(-\Delta + \frac{1}{4}\right) u = \left(-\frac{3}{4} + u^{p-1} \right) u \le 0 \]
			We take $v(r) := \frac{ r^{-\beta(d)} e^{-\frac{r}{2} } } { R^{-\beta(d)} e^{-\frac{R}{2} }}u(R)$, where $\beta(d) = 0$ for $d = 1, 2$ and $\beta(d) = \frac{d-1}{2}$ for $ d \ge 3$. It satisfies $\left( -\Delta + \frac{1}{4} \right)v \ge 0$ and $v(R) = u(R)$. The classical comparison theorem for $u$ and $v$ on $B_{R}^c$ implies an upper bound of $u$
			\beq u(r) \le v(r) = \frac{ r^{-\beta(d)} e^{-\frac{r}{2} } } { R^{-\beta(d)} e^{-\frac{R}{2} }}u(R), \quad r \ge R \ge R_0. \label{upu} \eeq
			
			Similarly, the positivity of $u$ implies 
			\[ (-\Delta +1) u = u^p \ge 0.\]
			And we compare $u$ on $B_{R}^c$ with a multiple of $r^{-\gamma(d)} e^{-r}$, where $\gamma(d) = \frac{d-1}{2}$ for $d \le 3$ and $\gamma(d) = d-2$ for $d \ge 4$ to ensure $(-\Delta + 1)(r^{-\gamma(d)} e^{-r} ) \le 0$. It results in a lower bound
			\beq u(r) \ge \frac{ r^{-\gamma(d)} e^{-r} } { R^{-\gamma(d)} e^{-r }}u(R), \quad r \ge R \ge R_0. \label{lou}\eeq
			
			Taking $\partial_r$ on the original elliptic equation ($\ref{semi}$) and using the radial symmetry of $u$, we get 
			\beq\left(-\Delta + 1 + \frac{d-1}{r^2} - pu^{p-1}\right) (-\partial_r u) = 0.\eeq
			From the definition of $R_0$, we have $1 + \frac{d-1}{r^2} - pu^{p-1} \in (\frac{1}{2}, 2)$ for $r \ge R_0$. So non-negativity of $-\partial_r u$ and similar comparison argument infers the following bounds,
			\begin{align}
				-\partial_r u(r) &\le \frac{ r^{-\gamma(d)} e^{-\frac{r}{2}} } { R^{-\gamma(d)} e^{-\frac{r}{2} }}(-\partial_ru)(R), \quad r\ge R \ge R_0,\label{upur} \\
				-\partial_r u(r)& \ge \frac{ r^{-\gamma(d)} e^{-2r} } { R^{-\gamma(d)} e^{-2r }}(-\partial_r u)(R),  \quad r\ge R \ge R_0. \label{lour}
			\end{align}
			Integrate (\ref{lour}) from $R$ to $+\infty$ and use Lemma \ref{lemexp},  
			\beq u(R) = \int_R^\infty (-\partial_r u)(r) dr \gtrsim_d (-\partial_r u)(R). \label{uur}\eeq
			
			\cu{Step 3. Controlling $r_{\d_0}$}.
			
			In this step, our goal is to prove an a priori bound for $r_{\d_0}$. Namely, there exists some $R(p, d) > 0$ (continuously depending on $p$) such that 
			\beq r_{\d_0} \le R(p, d). \label{step3} \eeq
			This part is essential for dealing with the unbounded domain.
			
			Without loss of generality, we assume $r_{\d_0} \ge d^{\frac{1}{2}}$ so that $R_0 = r_{\d_0}$ as (\ref{R}). The exponential decay bounds (\ref{upu}), (\ref{lou}), (\ref{upur}), (\ref{lour}) and (\ref{uur}) in Step 2 hold for $r \ge r_{\d_0}$. In particular,
			\[ -\partial_r u(r_{\d_0}) \lesssim_{d} u(r_{\d_0}) = \d_0 \sim_p 1.\] 
			From the upper bounds (\ref{upu}), (\ref{upur}), Lemma \ref{lemexp} and radially decreasing of $u$, we have
			\begin{align}
				\| u\|_{L^2(B_{r_{\d_0}})}^2 &\gtrsim_{d, p} r_{\d_0}^d, \label{louR} \\
				\| u\|_{L^2(B_{r_{\d_0}}^c)}^2 &\lesssim_{d, p} r_{\d_0}^{d-1}, \label{upuR} \\
				\| \nabla u\|_{L^2(B_{r_{\d_0}}^c)}^2 &\lesssim_{d, p} r_{\d_0}^{d-1}. \label{upurR}
			\end{align}
			
			Similar as Proposition \ref{propfuncid} for the Choquard equation, we get functional identities via inner product (\ref{semi}) respectively with $u$ 
			\[ \| \nabla u \|_{L^2(\real^d)}^2 +  \| u \|_{L^2(\real^d)}^2 - \| u \|_{L^{p+1}(\real^d)}^{p+1} = 0, \]
			and $x \cdot \nabla u$ (Pohozaev identity)
			\[ \frac{d-2}{2}\| \nabla u \|_{L^2(\real^d)}^2 +\frac{d}{2}\| u \|_{L^2(\real^d)}^2 - \frac{d}{p+1} \| u \|_{L^{p+1}(\real^d)}^{p+1} = 0.\]
			Eliminating the last term, we find
			\beq \| \nabla u \|_{L^2(\real^d)}^2 = \frac{d(p-1)}{(d+2) - p(d-2)} \| u \|_{L^2(\real^d)}^2 \sim_{d, p} \| u \|_{L^2(\real^d)}^2 \label{L2H1} \eeq
			
			Now we can find an $R_1 = R_1(d, p) \gg 1$ depending on constants in (\ref{louR})-(\ref{L2H1}), such that if $r_{\d_0} \ge R_1$, then 
			\[ 10 \| \nabla u \|_{L^2(B_{r_{\d_0}}^c)}^2 \le \| \nabla u \|_{L^2(\real^d)}^2 \sim_{d, p} \| u \|_{L^2(\real^d)}^2 \le \frac{11}{10} \|u\|_{L^2(B_{r_{\d_0}})}^2.   \]
			Thus the main contribution to $\dot{H}^1$ norm comes from $B_{r_{\d_0}}$.
			\beq  \| \nabla u\|_{L^2(B_{r_{\d_0}})}^2 \ge \frac{9}{10}\| \nabla u \|_{L^2(\real^d)}^2 \gtrsim_{d, p} r_{\d_0}^{d} \eeq
			However, we claim that the following two estimates in the connecting region $B_{r_{\d_0}} \backslash B_1$ and the center region $B_1$ hold.
			\begin{align}
				\| \nabla u\|_{L^2(B_{r_{\d_0}} \backslash{B_1} )}^2 &\lesssim_{d,p} r_{\d_0}^{d-1}, \label{H1middle} \\
				\| \nabla u\|_{L^2(B_1)}^2 &\lesssim_{d, p} 1 \label{H1center}.
			\end{align}
			These three estimates imply an a priori bound $R_2(d, p)$ for $r_{\d_0}$ depending on $(d, p)$, and take $R(d, p) := \max\{ d^{1/2}, R_1(d, p), R_2(d, p)\}$, we get the bound (\ref{step3}) and finish this step. 
			Next we prove these two estimates (\ref{H1middle}) and (\ref{H1center}). For simplicity, we denote 
			\[ y(r) := -\partial_r u(r).\]
			
			\underline{Step 3(a). Upper bound for $\dot{H}^1$ norm in the connecting region $B_{r_{\d_0}} \backslash B_1$}
			
			From the a priori bound (\ref{eigenbound1}) in Step 1 and radially decreasing property of $u$, we have
			\beq u(1) \lesssim_d \int_{B_1} u \lesssim_d Q(1) \int_{B_1} u \lesssim \int_{B_1} uQ \lesssim_{d, p} 1, \label{ubound}\eeq
			which implies 
			\beq \int_1^{\infty} y(r) dr = y(1) \lesssim_{d, p} 1.\label{L1y} \eeq
			We will see that this implies 
			\beq y(r) \lesssim_{d, p} 1,\quad \forall r \in [1, r_{\d_0}]. \label{ybound}\eeq 
			from the ODE evolution.
			Indeed, (\ref{semi}) and (\ref{ubound}) indicate that there exists $C_1(d, p)$ such that
			\[ \left| y'(r)+\frac{d-1}{r} y(r) \right|= \left|u(r) - u^p(r) \right| \le C_1(d, p),\quad r \in [1, \infty). \]
			Since $(r^{d-1} y)' = r^{d-1} (y' + \frac{d-1}{r} y)$, we integrate this inequality from $\tilde{r}$ to $r > \tilde{r} \ge 1$ to get
			\[ \left| r^{d-1} y(r) - \tilde{r}^{d-1} y(\tilde{r}) \right| \le \int_{\tilde{r}}^{r} s^{d-1} C_1(d, p) ds = \frac{C_1(d, p)}{d} (r^{d} - \tilde{r}^d) \le C_1(d, p) (r-\tilde{r})r^{d-1} \]
			and thus
			\beq y(r) \ge \left(\frac{\tilde{r}}{r}\right)^{d-1} y(\tilde{r}) - C_1(d, p) (r-\tilde{r}),\quad r > \tilde{r} \ge 1. \eeq
			So if for some $\tilde{r} \ge 1$ we have $y(\tilde{r}) \ge 2^{d}C_1(d, p)$, then for $r \in [\tilde{r}, \tilde{r}+1]$,
			\[ y(r) \ge 2^{-(d-1)}y(\tilde{r}) - C_1(d, p) \ge 2^{-d} y(\tilde{r}),\]
			and 
			\beq \int_1^\infty y(r) dr \ge \int_{\tilde{r}}^{\tilde{r} +1} y(r)dr \ge 2^{-d}y(\tilde{r}).\label{L1y2}\eeq
			(\ref{L1y}) and (\ref{L1y2}) confirm (\ref{ybound}). 
			
			(\ref{H1middle}) follows immediately the $L^1([1, r_{\d_0}])$ bound (\ref{L1y}) and $L^\infty([1, r_{\d_0}])$ bound (\ref{ybound}):
			\begin{align*}
				\| \nabla u\|_{L^2(B_{r_{\d_0}} \backslash{B_1} )}^2 \sim_d \int_1^{r_{\d_0}} y(r)^2 r^{d-1} dr \le r_{\d_0}^{d-1} \int_1^{r_{\d_0}} y(r) dr \left(\sup_{r\in [1, r_{\d_0}]} y(r)\right) \lesssim_{d, p} r^{d-1}_{\d_0}.
			\end{align*}

			\underline{Step 3(b). Upper bound for $\dot{H}^1$ norm in the center region $B_1$}
			
			Let $a:= u(1)$ and $v(r):= u(r) - a$, then $v$ is a radially decreasing, positive solution for the following elliptic problem on $B_1$:
			\beq \left\{ \begin{array}{ll} 
				\Delta v + g_p(v;a) = 0 & \text{in}\,\,B_1,\\
				v = 0 & \text{on}\,\,\partial B_1,
			\end{array} \right.
			\eeq
			where $g_p(v;a) := - (v+a) + (v+a)^p$.
			According to (\ref{ubound}) and (\ref{ybound}), $a = u(1)$, $-\partial_r u(1)$ are both bounded from above. So we can directly apply the method of \cite{de1982priori}, using subcriticality to verify (\ref{H1center}).
			
			Denote $$G_p(v;a) := \int_0^v g_p(s;a)ds = -\frac{\left( (v+a)^2 - a^2 \right)}{2}  +\frac{\left( (v+a)^{p+1}- a^{p+1} \right)}{p+1}.$$
			We still multiply this elliptic equation with $v$ and $x \cdot \nabla v$ respectively and integrate within $B_1$ to get 
			\begin{align}
				\| \nabla v \|_{L^2(B_1)}^2 - \int_{B_1} v g_p(v;a) dx  &= 0 \label{inner0} \\
				-(d-2)\| \nabla v \|_{L^2(B_1)}^2 + 2d \int_{B_1} G_p(v;a) dx &= 2\int_{\partial_{B_1}} |\nabla v(x)|^2 dS \label{poho0}
			\end{align}
			Eliminating $\| \nabla v \|_{L^2(B_1)}^2$, we get
			\beq 2d \int_{B_1} G_p(v;a)dx - (d-2) \int_{B_1}vg_p(v;a) dx = 2\int_{\partial_{B_1}} |\nabla v(x)|^2 dS \lesssim_{d, p} 1 \label{poho} \eeq
			Note that for $d \ge 3$, $p < \frac{d+2}{d-2}$. Comparing the highest order term of $v$ for $G_p(v;a)$ and $vg_p(v;a)$, it's easy to see that there exists $t_{d, p} \gg 1$ (continuously depending on $p$) such that for every $0 \le a = u(1) \lesssim_{d, p} 1$ as (\ref{ubound}) and $t \ge t_{d, p}$, 
			\begin{align}
				\left( \frac{d}{d-2} + \frac{p+1}{2} \right) G_p(t;a) - tg_p(t;a) \ge 0&,\nonumber \\ 
				G_p(t;a) \ge \frac{1}{2(p+1)} t^{p+1} & \label{posp1}.
			\end{align}
			Thus (\ref{poho}) implies 
			\begin{align*} \int_{B_1 \cap \{x: v(x) \ge t_{d,p} \}} \frac{\frac{d+2}{d-2} -p }{4\frac{p+1}{d-2}} v^{p+1}(x)  dx +    \int_{B_1 \cap \{x: v(x) < t_{d,p} \}} 
				2d G_p(v;a) - (d-2)vg_p(v;a) dx \lesssim_{d, p} 1   \end{align*}
			Then immediately we have
			\beq  \| v\|_{L^{p+1}(B_1)}^{p+1} \lesssim_{d, p} 1 \label{pboundin} \eeq
			for $d \ge 3$. If $d = 1, 2$, then (\ref{poho0}) directly implies 
			\[ \int_{B_1} G_p(v; a) dx \lesssim_{d, p} 1.\]
			So we similarly take a $t_{d, p}$ such that (\ref{posp1}) holds for all $a \lesssim_{d, p} 1$ as (\ref{ubound}) and $t \ge t_{d, p}$. This implies (\ref{pboundin}) for $d = 1, 2$.

			Now from (\ref{ubound}), (\ref{inner0}) and the form of $g_{p}(v; a)$, we get (\ref{H1center})
			\[ \| \nabla u \|_{L^2(B_1)}^2 =	\| \nabla v \|_{L^2(B_1)}^2 =  \int_{B_1} v g_p(v;a) dx \lesssim_{d, p} 1 +  \| v\|_{L^{p+1}(B_1)}^{p+1} \lesssim_{d, p} 1. \]
			This concludes Step 3. 
			
			\cu{Step 4. Concluding the proof.}
			
			Using the a priori bound for $r_{\d_0}$ (\ref{step3}) in Step 3, we estimate $\| u \|_{L^2(B_{R(d, p)}^c)} $ and $\| u \|_{L^2(B_{R(d, p)}^c)} $ respectively. 
			
			On the exterior region $B_{R(d, p)}^c$, we have exponential decay in Step 2 since $R_0 = \max\{r_{\d_0}, d^{\frac{1}{2}}\} \le R(d, p)$. Thus by taking $R=R(d, p)$ and integrate (\ref{upu}) with Lemma \ref{lemexp}, we see 
			\beq \| u \|_{L^2(B_{R(d, p)}^c)}^2 \lesssim_{d, p} R(d, p)^{d-1} \lesssim_{d, p} 1. \label{g2g} \eeq
			
			For the interior region $B_{R(d, p)}$, we try to use the argument of Step 3(b). First we bound the boundary values $u(R(d, p)) \lesssim_{d, p} 1$ and $-\partial_ru(R(d, p)) \lesssim_{d, p} 1$ from (\ref{eigenbound1}) and (\ref{uur}) respectively. Thereafter we can define $v := u - u(R(d, p))$ and use exactly the same argument as Step 3(b) to get 
			\[\| v \|_{L^{p+1} (B_{R(d, p)})} \lesssim_{d, p} 1.\]
			Finally,
			\begin{align} \|u \|_{L^2(B_{R(d, p)})} &\lesssim_d \| v \|_{L^2 (B_{R(d, p)})} + u(R(d, p)) R(d, p)^{\frac{d}{2}}  \nonumber\\
				& \lesssim_d \| v \|_{L^{p+1} (B_{R(d, p)})} R(d, p)^{d\left(\frac{1}{2} - \frac{1}{p+1}\right)} + u(R(d, p)) R(d, p)^{\frac{d}{2}} \lesssim_{d, p} 1.  \label{g3g}
			\end{align}
			
			Combine (\ref{g2g}) and (\ref{g3g}), we get the $L^2$ bound and therefore the $H^1$ bound via (\ref{L2H1}).

			\section{A priori estimates for the Choquard equation} \label{s32}
			
			In this section, we prove Theorem \ref{thmbound}, the a priori estimates for radially decreasing positive solutions of (\ref{Choquard}). The whole proof follows a similar framework, but differs in many ways due to the nonlinearity's nonlocal dependence on $u$. We define 
			\beq H_{d, \a, p}[u](x) := (|\cdot|^{-\a} * |u|^p)(x) |u|^{p-2}(x). \eeq
			For simplicity, we refer to it as $H[u](x)$ or even $H(x)$ if no ambiguity occurs. From Proposition \ref{propnonlocalevo}, $H[u]$, as a function of $x$, is positive and radially strictly decreasing. So we also define $s_{\l}$ by 
			\beq H[u](s_\l) = \l \eeq
			when  $\l \in (0, H[u](0))$ and $s_\l = 0$ when $\l \ge F[u](0)$.
			
			\cu{Step 1. A priori nonlinear eigenfunction estimate}
			
			We pick the same Schwartz, positive, radially decreasing $Q$ satisfying 
			\[ -\Delta Q + Q - Q^{p_0(d)} = 0 \]
			as in Step 1. of \S \ref{s31}. Then we have
			\beq \int H[u]u Q = \int (-\Delta + 1) u Q = \int u (-\Delta + 1) Q = \int u Q^{p_0} \label{intg2} \eeq
			Similarly, we define $C_0 := 2Q(0)^{p_0 - 1}$, and 
			\begin{align*}
				C_0 \int u Q &= \int_{H \le C_0} C_0 u Q + \int_{H \ge C_0} C_0 u Q \\
				&\le  \int_{B_{s_{C_0}^c}} C_0 u Q + \int H u Q 
				= \int_{B_{s_{C_0}^c}} C_0 u Q + \int u Q^{p_0} \\
				&\le\int_{B_{s_{C_0}^c}} C_0 u Q + \frac{C_0}{2} \int u Q.
			\end{align*}
			So we have
			\beq \int_{B_{s_{C_0}}} uQ \le \int_{B_{s_{C_0}^c}} u Q \eeq
			Using the monotonicity of $u$ this implies 
			\beq u(s_{C_0}) \int_{B_{s_{C_0}}} Q \le \int_{B_{s_{C_0}}} uQ \le \int_{B_{s_{C_0}^c}} u Q \le u(s_{C_0}) \int_{B_{s_{C_0}}^c} Q \nonumber \eeq
			This indicates an a priori bound for $s_{C_0}$ depending merely on $d$: there exists $R_1(d)$ such that
			\beq s_{C_0} \le R_1. \label{R1} \eeq

			%	     These maybe useless
			%	    \beq \int uQ \le 2 \int_{B_{s_{C_0}^c}} u Q \eeq
			%	    and by applying (\ref{intg2}) we also see
			%	    \beq \int FuQ \le 2C_0 \int_{B_{s_{C_0}^c}} u Q \eeq

			\cu{Step 2. Exponential decay far away.}
			
			This step resembles the one in \S \ref{s31} very much. We take
			\beq \d_0 := \min \left\{\frac 1 2, \frac{1}{2(p-1)} \right\},  \eeq
			and
			\beq R_0:= \max\{s_{\d_0}, d^{\frac{1}{2}} \}. \eeq
			Then for $r \ge R \ge R_0$, using the equation (\ref{Choquard})
			\[(-\Delta + 1) u = H[u]u \in (0, \frac{1}{2} u)  \]
			we get decay estimates from below and above by the classical elliptic comparison theorem
			\begin{align}
				u(r) \ge \frac{ r^{-\gamma(d)} e^{-r} } { R^{-\gamma(d)} e^{-r }}u(R), \quad r \ge R \ge R_0. \label{lou2} \\
				u(r) \le \frac{ r^{-\beta(d)} e^{-\frac{r}{2} } } { R^{-\beta(d)} e^{-\frac{R}{2} }}u(R), \quad r \ge R \ge R_0, \label{upu2}
			\end{align}
			where 
			\[ \beta(d) := \left\{ \begin{array}{ll} 0 & d = 1, 2, \\ \frac{d-1}{2} & d \ge 3, \end{array}  \right. \quad
			 \gamma(d) := \left\{ \begin{array}{ll} \frac{d-1}{2} & d = 1, 2, \\ d-2 & d \ge 3. \end{array}  \right.  \]
			
			Taking $\partial_r$ on (\ref{Choquard}), we find
			\beq\left(-\Delta + 1 + \frac{d-1}{r^2} - (p-1)H\right) (-\partial_r u) = -\partial_r \left( |\cdot|^{-\a} * |u|^p \right) u^{p-1} \ge 0. \eeq
			The non-negativity of the right hand side follows Proposition \ref{propnonlocalevo}. 
			So we get a lower bound of $-\partial_r u$
			\begin{align}
				-\partial_r u(r) &\ge \frac{ r^{-\gamma(d)} e^{-2r} } { R^{-\gamma(d)} e^{-2r }}(-\partial_ru)(R), \quad r\ge R \ge R_0,\label{lour2}
			\end{align}
			
			Again, integrate this lower bound (\ref{lour2}) from $R$ to $+\infty$ with Lemma \ref{lemexp}, we get
			\beq u(R) = \int_R^\infty (-\partial_r u)(r) dr \gtrsim_d (-\partial_r u)(R). \label{uur2}\eeq
			This also implies an exponential decay of $-\partial_r u(r)$
			\beq 	-\partial_r u(r) \lesssim_{d} \frac{ r^{-\beta(d)} e^{-\frac{r}{2} } } { R^{-\beta(d)} e^{-\frac{R}{2} }}u(R), \quad r \ge R \ge R_0, \label{upur2} \eeq
			
			\cu{Step 3. Controlling $s_{\d_0}$.}
			
			We want to obtain an an a priori bound for $s_{\d_0}$. Namely, there exists some $R(d, \a, p)  \sim_{d, \a, p} 1 $ (continuously depending on $(\a, p)$) such that 
			\beq s_{\d_0} \le R(d, \a, p). \label{step32} \eeq
			This time, we discuss the evolution of $u$ and $H[u]$ more carefully and utilize the structure of $H[u]$. First we derive pointwise control of $u$ on a large interval via information from $H[u]$. Then on such interval, the value of $u$ implies a non-trivial change in $H[u]$, which finally indicates a large variation in the evolution of $u$. This contradicts with the previous pointwise bound when $s_{\d_0}$ is too large.
			
			\underline{Step 3(a). Pointwise bound of $u$ on a large interval.}
			
			Firstly, we may assume $s_{\d_0} \ge \max\{R_1(d), d^{1/2}\}$ from (\ref{R1}). Next, apply (\ref{nonlocal1}), we see
			\[ \d_0 = H[u](s_{\d_0}) \sim_{d, \a} u(s_{\d_0})^{p-2} \left(  s_{\d_0}^{-\a} \int_0^{s_{\d_0}} u^p(s) s^{d-1} ds + \int_{s_{\d_0}}^\infty u^p(s) s^{d-1-\a}ds  \right).  \]
			Also note that from the monotonicity of $u$, exponential decay (\ref{upu2}) and Lemma \ref{lemexp} those two terms on the right hand side have different order of $s_{\d_0}$
			\begin{align} 
				s_{\d_0}^{-\a} \int_0^{s_{\d_0}} u^p(s) s^{d-1} ds & \gtrsim_{d} u^{p}(s_{\d_0}) s^{d-\a}_{\d_0}, \label{nonlocal2} \\
				\int_{s_{\d_0}}^\infty u^p(s) s^{d-1-\a}ds &\lesssim_{d, \a, p} u^{p}(s_{\d_0}) s^{d-\a-1}_{\d_0}. \label{nonlocal3}
			\end{align}
			So we have an $R_2 = R_2 (d, \a, p)$ and $C_1 = C_1(d, \a, p)$, such that when $s_{\d_0} \ge R_2$, \eqref{nonlocal2} dominates and thus
			\beq u(s_{\d_0})^{p-2}  s_{\d_0}^{-\a} \int_0^{s_{\d_0}} u^p(s) s^{d-1} ds \ge C_1. \label{order1} \eeq
			In this step, we assume $s_{\d_0} \ge \max\{R_2(d, \a, p), R_1(d), d^{1/2}\}$ later on. 	    
			
			Recall from Step 1 that $s_{C_0} \le R_1$. We apply (\ref{nonlocal1}) to see for every $r \ge R_1 \ge s_{C_0}$, 
			\[ C_0(d) \ge H[u](r) \gtrsim_{d, \a, p} u(r)^{p-2}  r^{-\a} \int_0^{r} u^p(s) s^{d-1} ds \gtrsim_{d} u(r)^{2p-2} r^{d-\a}. \]
			Hence there exists $C_2 = C_2(d, \a, p)$ and $C_3=C_3(d, \a, p)$, such that 
			\begin{align}
				u(r) \le C_2 r^{-\frac{d-\a}{2p-2}},\quad &\forall r \ge R_1, \label{pointwiseu} \\
				u(R_1)^{p-2}  R_1^{-\a} \int_0^{R_1} u^p(s) s^{d-1} ds &\le C_3. \nonumber
			\end{align}
			Take $M\in [1, \frac{s_{\d_0}}{R_1}]$ to be specified later, we have
			\begin{align}
				u^{p-2}(s_{\d_0}) \int_0^{\frac{s_{\d_0}}{M}} u^p(s) s^{d-1} ds &\le u(R_1)^{p-2} \int_0^{R_1} u^p(s) s^{d-1} ds + u\left(\frac{s_{\d_0}}{M} \right)^{p-2} \int_{R_1}^{\frac{s_{\d_0}}{M}} u^p(s) s^{d-1} ds \nonumber \\
				&\le C_3 + \left[ C_2 \left(\frac{s_{\d_0}}{M}\right)^{-\frac{d-\a}{2p-2}} \right]^{p-2} \int_{R_1}^\frac{s_{\d_0}}{M} \left(  C_2 s^{-\frac{d-\a}{2p-2}}\right)^{p} s^{d-1} ds \nonumber\\
				& \le C_3 + C_2^{2p-2} \left(d-\frac{(d-\a)p}{2p-2}\right)^{-1} \left(\frac{s_{\d_0}}{M}\right)^{\a}, \label{order2}
			\end{align}
			where we used $p \ge 2 > \frac{2d-\a}{d}$, so that $-\frac{d-\a}{2p-2}p + d = \frac{p(d+\a) - 2d}{2p-2} > 0$.
			Now take $M(d, \a, p) := \left( 4C_1 C_2^{2p-2} \left(d-\frac{(d-\a)p}{2p-2}\right)^{-1} \right)^{\frac{1}{\a}}$ and $R_3(d, \a, p) := \max\{R_1 M, (4C_1C_3)^{\frac{1}{\a}} \}$. If $s_{\d_0} \ge \max \left\{ R_1, R_2, R_3, d^{1/2} \right\}$, then (\ref{order1}) and (\ref{order2}) imply
			\beq \int_{\frac{s_{\d_0}}{M}}^{s_{\d_0}} u^{2p-2}(s) s^{d-1} ds  \ge u^{p-2}(s_{\d_0}) \int_{\frac{s_{\d_0}}{M}}^{s_{\d_0}} u^p(s) s^{d-1} ds \ge \frac{1}{2}C_1 s_{\d_0}^{\a}. \label{order3} \eeq
			Combining this integral lower bound with the pointwise upper bound (\ref{pointwiseu}), we get a pointwise lower bound on a large subset of $[M^{-1} s_{\d_0}, s_{\d_0}]$: there exists an $\mu(d,\a,p) \in (0, 1)$ and $C_4(d, \a, p) > 0$, such that 
			\[ \left| \left\{ r \in [M^{-1} s_{\d_0}, s_{\d_0}] : u(r) \ge C_4 r^{-\frac{d-\a}{2p-2}} \right\} \right| \ge \mu s_{\d_0}. \]
			From the monotonicity of $u$, we have
			\beq u(r)\sim_{d, \a, p} s_{\d_0}^{-\frac{d-\a}{2p-2}},\quad r \in [M^{-1}s_{\d_0}, (M^{-1} + \mu) s_{\d_0}]. \eeq 
			We denote $[T_0, T_1] := [M^{-1}s_{\d_0}, (M^{-1} + \mu) s_{\d_0}]$, and stress that 
			\[ T_1 -T_0 \sim_{d, \a, p} s_{\d_0}. \]
			
			\underline{Step 3(b). Non-trivial evolution of $H[u]$.}
			
			Next, we discuss the evolution of $u$ and $H[u]$ on this region to arrive at an estimate of $s_{\d_0}$. 
			Denote $y(r):= -\partial_r u(r)$. Since $(r^{d-1} y)' = r^{d-1} (y' + \frac{d-1}{r} y)$, we multiply $r^{d-1}$ and integrate (\ref{Choquard}) to get
			\beq r_2^{d-1} y(r_2) - r_1^{d-1}y(r_1) = \int_{r_1}^{r_2} \left( H[u](s) - 1\right)u(s) s^{d-1}ds, \quad 0 < r_1 \le r_2 \label{evo} \eeq
			For $r \in [T_0, T_1],$ we apply (\ref{nonlocalevo}) with $\epsilon_0 := (4M)^{-1}$ and the exponential decay (\ref{upu2}) to get
			\beq -\frac{d}{dr}\left( |\cdot|^{-\a} * u^p \right)(r) \gtrsim_{d, \a, p} \left[ u^p\left({\frac{2r}{3}}\right) - u^p(2Mr)  \right] r^{d-1-\a} \gtrsim_{d, \a, p} s_{\d_0}^{-\frac{d-\a}{2(p-1)}p+d-1-\a}. \eeq 
			Thus for any $r \in [T_0, \frac{T_0 + 7T_1}{8}]$, we have
			\begin{align*}
				 &\quad H[u](r) - H[u]\left(r+\frac{T_1-T_0}{8}\right) \\
				 &= u^{p-2}(r)\left(|\cdot|^{-\a} * u^p\right)(r) - u^{p-2}\left(r+\frac{T_1-T_0}{8}\right) \left(|\cdot|^{-\a} * u^p\right) \left(r+\frac{T_1-T_0}{8}\right)  \\
				 &\ge u^{p-2}(r) \left[ \left(|\cdot|^{-\a}* u^p\right)(r)- \left(|\cdot|^{-\a}* u^p\right) \left(r+\frac{T_1-T_0}{8}\right)  \right] \\
				 &\ge u^{p-2}(T_0) \int_r^{r+\frac{T_1-T_0}{8}} \frac{d}{ds}({-|\cdot|^{-\a} * u^p})(s)ds \\
				 &\gtrsim_{d, \a, p} s_{\d_0}^{-\frac{d-\a}{2p-2}(p-2)} (T_1-T_0) s_{\d_0}^{-\frac{d-\a}{2(p-1)}p + d - 1 - \a}  \sim_{d, \a, p} 1
 			\end{align*}
			To be more specific, there exists an $\d_1(d, \a, p) > 0$ such that 
			\beq H(r) - H\left(r+\frac{T_1-T_0}{8}\right) \ge \d_1,\quad r \in [T_0, \frac{T_0 + 7T_1}{8}]. \label{Fvar} \eeq 
			
			\underline{Step 3(c). Large evolution of $u$.}
			
			Now we divide into two cases with respect to the position of $s_1$, namely $F(s_1) = 1$, and derive an a priori bound of $s_{\d_0}$ in each case.
			
			\emph{Case 1. $s_1 \le \frac{T_0+T_1}{2}$.}  
			
			(\ref{Fvar}) implies that for $r \ge \frac{3T_0+5T_1}{8}$,
			\[ H(r) \le H\left(r-\frac{T_1-T_0}{8} \right) - \d_1 \le H(s_1)-\d_1 = 1-\d_1, \]
			The evolution estimate (\ref{evo}) then indicates a lower bound for $y(r)$ with $r_1 = r \in [\frac{3T_0+5T_1}{8}, \frac{T_0+3T_1}{4}]$ and $r_2 = T_1$:
			\[ y(r) \ge r^{-(d-1)}\int_{r}^{T_1}(1-H(s))u(s)s^{d-1} ds \gtrsim_{d, \a, p} s_{\d_0}^{-(d-1) + 1 -\frac{d-\a}{2(p-1)} + d-1} = s_{\d_0}^{1-\frac{d-\a}{2(p-1)}}.  \]
			So we integrate it to get
			\[  u\left( \frac{3T_0+5T_1}{8}\right) \ge \int_{\frac{3T_0+5T_1}{8}}^{\frac{T_0+3T_1}{4}} y(r)dr \gtrsim_{d. \a, p} s_{\d_0}^{2-\frac{d-\a}{2(p-1)}}. \]
			This implies a bound $s_{\d_0} \le R_4(d, \a, p)$ from (\ref{pointwiseu}).
			
			\emph{Case 2. $s_1 \ge \frac{T_0+T_1}{2}$.}
			
			Conversely, we focus on the interval $[T_0, \frac{T_0+T_1}{2}]$ in this case. For $r \in [T_0, \frac{5T_0 + 3T_1}{8}]$, 
			\[ H(r) \ge H\left(r+\frac{T_1-T_0}{8} \right) + \d_1 \ge H(s_1)+\d_1 = 1+\d_1, \]
			And for $r \in [\frac{3T_0+T_1}{4}, \frac{5T_0+3T_1}{8}]$, we take $r_1 = T_0$, $r_2 =r$ and apply (\ref{evo}) 
			\[ y(r) \ge r^{-(d-1)}\int_{T_0}^{r} (H(s) - 1)u(s)s^{d-1} ds \gtrsim_{d, \a, p} s_{\d_0}^{-(d-1) + 1 -\frac{d-\a}{2(p-1)} + d-1} = s_{\d_0}^{1-\frac{d-\a}{2(p-1)}}.  \]
			Similarly, we integrate that on $[\frac{3T_0+T_1}{4}, \frac{5T_0+3T_1}{8}]$ to get a lower bound of $u\left( \frac{3T_0+5T_1}{8}\right)$. Combining with (\ref{pointwiseu}), this provides us with a bound $s_{\d_0} \le R_4'(d, \a, p)$. 
			
			So now we can conclude this step by taking $R(d, \a, p) := \max\{d^{\frac{1}{2}}, R_1, R_2, R_3, R_4, R_4' \}$ in (\ref{step32}).
			
			\cu{Step 4. A priori $L^2$ bound.}
			
			In this step, we derive an a priori bound of $L^2(\real^d)$ norm
			\beq \| u \|_{L^2(\real^d)} \lesssim_{d, \a, p} 1 \label{step42}\eeq
			with its constant depending continuously on $(\a, p)$. We control $\|u\|_{L^2(B_{2R}^c)}$ and $\|u\|_{L^2(B_{2R})}$ respectively with $R$ from (\ref{step32}).
			
			The control on exterior region follows directly from the exponential decay (\ref{upu2}) and the pointwise control (\ref{pointwiseu}). Moreover, from (\ref{uur2}), we also get $-\partial_r u({2R}) \lesssim_d u(2R) \lesssim_{d, \a, p} 1$.
			
			Regarding the interior region $B_{2R}$, we use the local argument similar to Step 3(b) in \S \ref{s31} (originated in \cite{de1982priori}), but more involved due to the essentially nonlocal nature of the nonlinearity $H[u]$. 
			\begin{rmk}\label{p2s}
				The structure of $H[u]$ (specifically, (\ref{nonlocal1})) also implies a simple way to obtain the desired bound when $p=2$. Indeed,
				\[ \d_0 \ge H[u]({2R}) = (|\cdot|^{-\a} * u^2)(2R) \gtrsim_{d, \a} R^{-\a} \int_{B_{2R}} u^2 dx \sim_{d, \a, p} \| u\|_{L^2(B_{2R})}^2. \]
				However, this argument does not work for $p > 2$ due to the lack of lower bound of $u({2R})$. Our following argument solves this problem and presents a uniform control for $p \ge 2$ as well. 
			\end{rmk}
			Denote $a := u({2R})$. We first claim the following two functional identities (the local version of \eqref{func01} and \eqref{func02})
			\begin{align}
				\int_{B_{2R}} \left(|\nabla u|^2 + u^2 - H[u]u^2 \right)dx &= \uppercase\expandafter{\romannumeral1} \label{func1}\\
				\int_{B_{2R}} \left(-\frac{d-2}{2} |\nabla u|^2 -\frac{d}{2} u^2 + \frac{2d-\a}{2p} H[u]u^2\right)dx &= \uppercase\expandafter{\romannumeral2}_1+
				\uppercase\expandafter{\romannumeral2}_2 +
				\uppercase\expandafter{\romannumeral2}_3 \label{func2}	    	\end{align}
			where 
			\begin{align*}
				\uppercase\expandafter{\romannumeral1}& =  	a \int_{\partial B_{2R}} \partial_r ud\sigma, \\
				\uppercase\expandafter{\romannumeral2}_1 & = {2R}\int_{\partial B_{2R}} \left(\frac{1}{2} |\nabla u|^2  +  \frac{1}{p}H[u]u^2 - \frac{1}{2}a^2\right) d\sigma  \\
				\uppercase\expandafter{\romannumeral2}_2  &= - \frac{\a}{2p} \int_{B_{2R}} \int_{B_{2R}^c} |x-y|^{-\a} u^p(y) u^p(x) dydx\nonumber \\
				\uppercase\expandafter{\romannumeral2}_3  &= \frac{\a}{p} \int_{B_{2R}} \int_{B_{2R}^c} |x-y|^{-\a} u^p(y) u^p(x) \frac{(x-y) \cdot x}{|x-y|^2} dydx
			\end{align*}
			Indeed, (\ref{func1}) and (\ref{func2}) come from multiplying (\ref{Choquard}) with $u$ and $x \cdot \nabla u$ respectively, integrating on $B_{2R}$ and using integration by parts. The only tricky point is to apply the symmetry 
			\begin{align*} \int_{B_{2R}} \int_{B_{2R}} |x-y|^{-\a} u^p(y) u^p(x) \frac{(x-y) \cdot x}{|x-y|^2} dydx &= \int_{B_{2R}} \int_{B_{2R}} |x-y|^{-\a} u^p(y) u^p(x) \frac{(y-x) \cdot y}{|x-y|^2} dydx \\
				&= \frac{1}{2} \int_{B_{2R}} \int_{B_{2R}} |x-y|^{-\a} u^p(y) u^p(x) dydx. 
			\end{align*}
			to derive (\ref{func2}). 
			
			Next, we claim that right hand sides of (\ref{func1}) and (\ref{func2}) are bounded for our $u$.
			\begin{claim}
				\beq |\uppercase\expandafter{\romannumeral1}| +|\uppercase\expandafter{\romannumeral2}_1| +|\uppercase\expandafter{\romannumeral2}_2|
				+|\uppercase\expandafter{\romannumeral2}_3| \lesssim_{d, \a ,p} 1.  \eeq
			\end{claim}
			We postpone its proof and see how this implies the $L^2(B_{2R})$ bound. Eliminating $|\nabla u|^2$ in (\ref{func1}) and (\ref{func2}), we get from the claim that
			\beq \left| \int_{B_{2R}} \left[ \frac{2d - \a - p(d-2)}{2p}H[u] - 1 \right]u^2 dx \right| \lesssim_{d,\a,p}1 \label{mismatch} \eeq
			Note that $2d - \a - p(d-2) > 0$ follows from (\ref{para2}).
			Take $A:= \frac{4p}{2d - \a - p(d-2)}$. If $s_A \ge 2R$, (\ref{mismatch}) yields the desired $L^2$ bound on $B_{2R}$. So we consider $s_A \in [0, {2R})$, in which case \eqref{mismatch} indicates
			\beq \int_{B_{s_A}} u^2dx \lesssim_{d, \a, p} 1 \label{sa} \eeq
			We further discuss $u(s_A)$ to control $L^2$ norm on $B_{2R} \backslash B_{s_A}$:
			\begin{itemize}
				\item If $u(s_A) \le 1$, then 
				\[ \int_{B_{2R} \backslash B_{s_A}} u^2 dx \le \int_{B_{2R} \backslash B_{s_A}} dx \lesssim_{d, \a, p} 1 \]
				\item If $u(s_A) \ge 1$, then the argument of Remark \ref{p2s} works.
				\begin{align*} A &= (|\cdot|^{-\a} * u^p)(s_A) u^{p-2}(s_A) \ge (|\cdot|^{-\a} * u^p)(s_A)  \\
					&\gtrsim_{d, \a} s_A^{-\a} \int_{0}^{s_{A}} u^p s^{d-1} ds + \int_{s_A}^{2R} u^p s^{d-1-\a} ds \\
					&\ge (2R)^{-\a} \int_{B_R} u^p dx \gtrsim_{d, \a, p} \|u\|_{L^p(B_{2R})}^p \gtrsim_{d, \a, p} \|u\|_{L^2(B_{2R})}^p.
				\end{align*}
			\end{itemize}
			These discussions and (\ref{sa}) imply (\ref{step42}). 
			
			To end this step, we now prove the claim. $\uppercase\expandafter{\romannumeral1}$ and $\uppercase\expandafter{\romannumeral2}_1$ are bounded from the pointwise bound of $u({2R})$ (by \eqref{pointwiseu}), $-\partial_r u({2R})$ (by \eqref{upur2}) and $H[u](2R) \le H[u](s_{\d_0}) \le \d_0$. For $\uppercase\expandafter{\romannumeral2}_2$, we first check
			\beq
			\begin{split}
				\sup_{x \in B_{2R}} \int_{B_{2R}^c} u^2(y) |x-y|^{-\a}dy&\le \sup_{x \in B_{2R}} \int_{B_{4R}^c}  u^2(y) |x-y|^{-\a}dy + \int_{B_{4R} \backslash B_{2R}} u^2(y) |x-y|^{-\a}dy \\
				&\lesssim_d \int_{4R}^\infty u^2(r)\left(\frac{r}{2}\right)^{-\a}r^{d-1} dr + u^2({2R}) \int_{B_{6R}} |z|^{-\a} dz 
				\lesssim_{d, \a, p} 1,
			\end{split} \label{unif1}
			\eeq
			where the last inequality follows the exponential decay (\ref{upu2}) and $\a < d$. This yields
			\begin{align*}
				\left|\uppercase\expandafter{\romannumeral2}_2 \right| 
				\le \frac{\a}{2p} u^{p-2}({2R})\int_{B_{2R}} u^p(x)dx\left[ \sup_{x \in B_{2R}} \int_{B_{2R}^c} u^2(y) |x-y|^{-\a}dy \right] \lesssim_{d, \a, p} 1
			\end{align*}
			where we also apply (\ref{nonlocal1}) to obtain $u^{p-2}({2R})\int_{B_{2R}} u^p(x)dx \lesssim_{d, \a} R^\a H[u]({2R}) \le R^\a \d_0$.
			
			Finally, we will control $\uppercase\expandafter{\romannumeral2}_3$. We integrate by parts to avoid the weight $|x-y|^{-\a-1}$ which is not integrable when $\a \ge d-1$
			\begin{align*}
				p \uppercase\expandafter{\romannumeral2}_3 =& \sum_{i=1}^d \int_{B_{2R}} x_i u^p(x) \int_{B_{2R}^c} \partial_{y_i} |x-y|^{-\a} u^p(y) dydx \nonumber\\
				= &\sum_{i=1}^d \int_{B_{2R}} x_i u^p(x) \int_{B_{2R}^c} \left[\partial_{y_i} \left( |x-y|^{-\a} u^p(y) \right) - |x-y|^{-\a} \partial_{y_i} \left(u^p(y) \right) \right]  dydx \nonumber\\
				=& - \sum_{i=1}^d \int_{B_{2R}} x_i u^p(x) \int_{\partial B_{2R}} |x-y|^{-\a} u^p(y) e_i\cdot d\sigma dx \\
				& - p\sum_{i=1}^d \int_{B_{2R}} x_i u^p(x) \int_{B_{2R}^c}   |x-y|^{-\a} u^{p-1}(y)  \partial_i u(y)  dy dx  := \uppercase\expandafter{\romannumeral2}_{31} + \uppercase\expandafter{\romannumeral2}_{32}
			\end{align*}
			where $e_i$ is the $i$-th vector of the standard basis. 
			
			Similar as (\ref{unif1}), with the exponential decay of $-\partial_r u$ from (\ref{upur2}), we have 
			\beq \sup_{x \in B_{2R}} \int_{B_{2R}^c} |x-y|^{-\a} u(y)|\nabla u(y)| dy\lesssim_{d, \a, p} 1. \label{unif2} \eeq
			Then
			\[ |\uppercase\expandafter{\romannumeral2}_{32}|\lesssim_{d,\a, p} 1 \]
			follows in the same way as controlling $\uppercase\expandafter{\romannumeral2}_2$. 
			And $\uppercase\expandafter{\romannumeral2}_{31}$ comes as
			\begin{align*}
				|\uppercase\expandafter{\romannumeral2}_{31}|& \le 2dR \left[ \int_{B_R}  u^p(x) \int_{\partial B_{2R}} |x-y|^{-\a} u^p({2R}) d\sigma dx + \int_{B_{2R}\backslash B_R} u^p(x) \int_{\partial B_{2R}} |x-y|^{-\a} u^p({2R}) d\sigma dx  \right]  \\
				&\le  2dR \left[ \int_{B_R}  u^p(x)dx R^{-\a} u^p({2R})\int_{\partial B_{2R}}  d\sigma  + u^p(R)u^p(2R) \int_{B_{4R}} |z|^{-\a} dz   \int_{\partial B_{2R}} d\sigma \right] \\
				&  \lesssim_{d,\a,p} H[u](R) u^{2}(2R) + u^p(R)u^p(2R)     \lesssim_{d, \a, p} 1.
			\end{align*}
			This finishes the proof of that claim and this step.

			\cu{Step 5. Concluding the proof.}
			
			In this final step, we use the $L^2$ a priori bound (\ref{step42}) to verify (\ref{aprioribound}) and (\ref{aprioribound2}). 
			
			The $L^2$ bound implies an $H^1$ bound immediately from Proposition \ref{propfuncid}. Then for given $r \in (1, \infty)$, we apply the proof of \cite[Proposition 4.1]{moroz2013groundstates}, the classic bootstrap method for semilinear elliptic problems plus the Hardy-Littlewood-Sobolev inequality to improve the regularity of solutions. The $H^1$ bound therefore implies a $W^{2,r}$ bound (\ref{aprioribound}) after finite times of bootstrap iterations. The uniformity of $(\a, p, r)$ in \eqref{aprioribound} comes from the uniformity of constants in the Hardy-Littlewood-Sobolev inequality and the elliptic estimates in every iteration. 
			
			%Note that when $(\a, p)$ are taken values in a compact admissible (namely satisfying (\ref{para2})) set $K$, the number of iteration times can be uniformly bounded, and the right hand side at $n$-th iteration can also be uniformly bounded by $C_n(d, K, r) \| u\|_{H^1}^{N_n(d, K, r)}$. So the uniformity is also guaranteed. 
			
			As for (\ref{aprioribound2}), we use Sobolev embedding to get a $C^1(\real^d)$ bound from (\ref{aprioribound}). Then we get pointwise bound of $u(R)$ from (\ref{pointwiseu}) and (\ref{step32}), and thus obtain exponential decay of $u(r)$, $|\nabla u(r)| = -\partial_r u(r)$ when $r\ge R(d, \a, p)$ from  (\ref{upu2}) and (\ref{upur2}). These two facts yield (\ref{aprioribound2}) and finish the whole proof of Theorem \ref{thmbound}. 
			
			\section{Non-degeneracy and Uniqueness}\label{s4}
			
			In this section, we prove the non-degeneracy and uniqueness of positive solution for $(\a, p) \in [d-2-\d, d-2+\d] \times [2, 2+\d]$ with $d \in \{ 3, 4, 5\}$, $\d \ll 1$. The  starting point will be non-degeneracy and uniqueness for the Newtonian case $(\a, p) = (d-2, 2)$ with $d \in \{ 3, 4, 5\}$. We refer the reader to \cite{chen2021nondegeneracy} for non-degeneracy and \cite[Appendix A]{arora2019global} for uniqueness, also to \cite{lieb1977existence,lenzmann2009uniqueness} for the original proof for $d = 3$.
			
			\subsection{Compactness analysis}\label{s51}
			
			As a preparation, we first establish the following compactness result for radial positive solutions. It makes use of the a priori bound Theorem \ref{thmbound}.
			
			\begin{prop}\label{propasym}
				Let $d \in \{3, 4, 5\}$, there exists $\delta > 0$ such that the following holds. For any sequence $\{(\a_n, p_n) \} \subset [d-2 - \delta, d-2+\d] \times [2, 2+\d]$ with $( \alpha_n, p_n) \rightarrow (d-2, 2)$ when $n \rightarrow \infty$, and $Q_n$ be an $H^1$ radial positive solution for (\ref{Choquard}) with parameters $(d, \a_n, p_n)$. Then  
				\[ Q_n \rightarrow Q_0 \quad \text{in}\,\,H^1(\real^d)\cap L^\infty(\real^d)\]
				where $Q_0$ is the unique radial positive solution for $(d, \a, p) = (d, d-2, 2)$.
				
			\end{prop}
			\begin{proof}
				It suffices to show that for any sequence $\{Q_n\}$ as above, we can find subsequence $Q_{n_k} \rightarrow Q_0$ in $H^1$ and $L^\infty$. Before starting, we take $\d$ small enough such that $(\a_n, p_n)$ satisfies (\ref{para2}).
				
				\cu{Step 1. Bounds and convergences.}
				
				Theorem \ref{thmbound} and the range of $(\a_n, p_n)$ indicate a uniform upper bound
				\beq \| Q_n \|_{H^1} + \| Q_n\|_{L^\infty} \lesssim_{d, \delta} 1 . \label{Linfbound} \eeq 
				Note that $[\frac{11}{5}, 3] \subset (2, \frac{2d}{d-2})$ for $d \in \{3, 4, 5\}$. Thus by the compact embedding $H^1_{rad}(\real^d) \hookrightarrow L^q_{rad}(\real^d)$ for $q \in (2, \frac{2d}{d-2})$, there exist $Q \in H^1_{rad}(\real^d)$ and a subsequence (still denoted by $\{Q_n\}$) such that 
				\beq Q_n \rightharpoonup Q \quad \text{in} \,\,H^1;\qquad Q_n \rightarrow Q \quad \text{in} \,\,L^{\frac{11}{5}} \cap L^3. \label{H1conv} \eeq
				A further subsequence (still denoted by $\{Q_n\}$) implies almost everywhere convergence
				\beq Q_n \rightarrow Q\quad a.e.  \label{aeconv}\eeq
				
				With these bounds and convergences, we claim the following convergence.
				\begin{claim} For any uniformly bounded $L^\frac{4d}{d+2}$ functions $\{\phi_n\}$ and $\phi$ with $\phi_n \rightarrow \phi$ in $L^{\frac{4d}{d+2}}$, we have
					\beq \int_{\real^d} \left(|\cdot|^{-\a_n}* Q_n^{p_n} \right)Q_n^{p_n -1} \phi_n dx\rightarrow \int_{\real^d} \left(|\cdot|^{-(d-2)}* Q^{2} \right)Q \phi dx \label{nonlocalconv}\eeq
				\end{claim}
				
				\cu{Step 2. End of the proof with the claim.}
				
				Postponing the proof of (\ref{nonlocalconv}) to Step 3, we finish the proof of this proposition with that convergence. Firstly, for any $\phi \in H^1(\real^d)$, we take $\phi_n = \phi$ for all $n \in \posint$ and use (\ref{H1conv}) and (\ref{nonlocalconv}), then
				\begin{align*} &\int_{\real^d}\left[\nabla Q \cdot \nabla \phi + Q\phi -   \left(|\cdot|^{-(d-2)}* Q^{2} \right)Q\phi \right] dx \\
					=& \lim_{n\rightarrow \infty} \int_{\real^d}\left[ \nabla Q_n \cdot \nabla \phi + Q_n\phi-  \left(|\cdot|^{-\a_n}* Q_n^{p_n} \right)Q_n^{p_n -1}\phi  \right] dx= 0. \end{align*}
				Hence $Q$ is an $H^1$ radial solution of (\ref{Choquard}) with $(d, \a, p) = (d, d-2, 2)$.Taking $\phi_n := Q_n \rightarrow Q =:\phi$ in $L^\frac{4d}{d+2}$, which follows (\ref{H1conv}) and that $\frac{11}{5} < \frac{4d}{d+2} < 3$ for $d \in \{3, 4, 5\}$, we have
				\beq \int_{\real^d} \left(|\cdot|^{-\a_n}* Q_n^{p_n} \right)Q_n^{p_n} dx \rightarrow \int_{\real^d}\left(|\cdot|^{-(d-2)}* Q^{2} \right)Q^2 dx. \label{nonlconv} \eeq
				Also by Corollary \ref{corolo} and (\ref{func01}), for $(\a_n, p_n) \in [d-2-\d, d-2+\d] \times [2, 2+\d]$, we have a uniform lower bound 
				\[\int_{\real^d}  \left(|\cdot|^{-\a_n}* Q_n^{p_n} \right)Q_n^{p_n} dx =  \| Q_n\|_{H^1}^2 \gtrsim_{d,\d} 1. \]
				So the limit in (\ref{nonlconv}) is nonzero and $Q$ is nonzero. Notice that (\ref{aeconv}) ensures that $Q$ is also non-negative and, moreover, positive from the strong maximal principle. So far, we have verified $Q$ to be a positive and radial $H^1$ solution for (\ref{Choquard}) with $(d, \a, p) = (d, d-2, 2)$. Thus $Q = Q_0$ by uniqueness.

				Together with (\ref{func03}) , (\ref{func04}) and $(\a_n, p_n) \rightarrow (d-2, 2)$, we see
				\[ \| Q_n\|_{L^2}^2 \rightarrow \| Q\|_{L^2}^2,\qquad \| \nabla Q_n\|_{L^2}^2 \rightarrow \| \nabla  Q\|_{L^2}^2. \]
				Thus (\ref{H1conv}) can be improved to be strong $H^1$ convergence. 
				
				As for $L^\infty$ convergence, we use similar strategy as in Step 5 of \S \ref{s32} to improve the regularity. Take the difference of equations of $Q_n$ and $Q_0$, we have 
				\beq (-\Delta + 1)(Q_n - Q_0) = (|\cdot|^{-\a_n} * Q_n^{p_n} ) Q_n^{p_n-1} - (|\cdot|^{-(d-2)} * Q_0^{2} ) Q_0.\label{diff} \eeq
				Using \eqref{contineq1} in Lemma \ref{lemdiff} with $u_0 = Q_0$, $u = Q_n$, $(\a, p) = (\a_n, p_n)$, we get from $\| Q_n -Q_0\|_{H^1} = o_{n} (1)$ that
				\[ \|(|\cdot|^{-\a_n} * Q_n^{p_n} ) Q_n^{p_n-1} - (|\cdot|^{-(d-2)} * Q_0^{2} ) Q_0 \|_{L^\frac{2d}{d-2}} = o_{n}(1).\] 
				Then since $W^{2, \frac{2d}{d-2}}(\real^d) \hookrightarrow L^\infty(\real^d)$ when $d < 6$, 
				\eqref{diff} implies 
				\[ \| Q_n - Q_0\|_{L^\infty} \lesssim_d  \| Q_n - Q_0\|_{W^{2, \frac{2d}{d-2}}} \lesssim_d\|(|\cdot|^{-\a_n} * Q_n^{p_n} ) Q_n^{p_n-1} - (|\cdot|^{-(d-2)} * Q_0^{2} ) Q_0 \|_{L^\frac{2d}{d-2}} = o_{n}(1)  \]
				which is our desired $L^\infty$ convergence.
				
				\cu{Step 3. Proof of the claim.}
				
				Finally, we verify the claim. Decompose
				\[\int_{\real^d} \left(|\cdot|^{-\a_n}* Q_n^{p_n} \right)Q_n^{p_n -1} \phi_n dx - \int_{\real^d} \left(|\cdot|^{-(d-2)}* Q_n^{2} \right)Q \phi dx := \uppercase\expandafter{\romannumeral1}+\uppercase\expandafter{\romannumeral2} + \uppercase\expandafter{\romannumeral3} \]
				where
				\begin{align*}
					\uppercase\expandafter{\romannumeral1} &:=  \int_{\real^d} \left(|\cdot|^{-\a_n}* Q_n^{p_n} \right)(Q_n^{p_n -1} \phi_n - Q\phi) dx, \\
					\uppercase\expandafter{\romannumeral2} &:=  \int_{\real^d} \left(|\cdot|^{-\a_n}* (Q_n^{p_n} - Q^2) \right)Q\phi dx, \\						
					\uppercase\expandafter{\romannumeral3} &:=  \int_{\real^d} \left((|\cdot|^{-\a_n} - |\cdot |^{-(d-2)} )* Q^2 \right)Q \phi dx.
				\end{align*}
				Then we may restrict $\d$ to be smaller to estimate each term on the right hand side.
				
				For $\uppercase\expandafter{\romannumeral1}$,
				\begin{align*}
					|\uppercase\expandafter{\romannumeral1}| &\le \left| \int_{\real^d} \left(|\cdot|^{-\a_n}* Q_n^{p_n} \right)(Q_n^{p_n -1} - Q)\phi_n dx \right| + \left| \int_{\real^d} \left(|\cdot|^{-\a_n}* Q_n^{p_n} \right)Q (\phi_n - \phi) dx \right| \\
					& \le \||\cdot|^{-\a_n}* Q_n^{p_n}  \|_{L^\frac{2d}{d-2}} \left[ \| Q_n^{p_n - 1} - Q \|_{L^\frac{4d}{d+2}} \| \phi_n \|_{L^ \frac{4d}{d+2}} + \| Q \|_{L^\frac{4d}{d+2}} \| \phi_n - \phi \|_{L^\frac{4d}{d+2}}  \right].
				\end{align*}
				Using Hardy-Littlewood-Sobolev inequality,
				\begin{align}
					\||\cdot|^{-\a_n}* Q_n^{p_n}  \|_{L^\frac{2d}{d-2}} \lesssim_{d, \d} \| Q_n\|_{L^{\frac{2dp_n}{3d-2-2\a_n}}}^{p_n} \lesssim_{d, \d} \| Q_n \|_{H^1}^{p_n} \lesssim_{d, \d} 1  \label{convsmall}
				\end{align}
				where $\frac{2dp_n}{3d-2-2\a_n} = \frac{4d}{d+2} + o_{\d}(1)$ so that we can take $\d$ small enough such that 
				\beq \frac{2dp}{3d-2-2\a} \in \left[\frac{11}{5}, 3\right]\quad \forall (\a, p) \in [d-2-\d, d-2+\d] \times [2, 2+\d].\label{d1} \eeq 
				Also, we introduce $T \gg 1$ to estimate
				\begin{align*}
					&\,\,\,\,\,\,\,\,\| Q_n^{p_n - 1} - Q \|_{L^\frac{4d}{d+2}} \le \| Q_n^{p_n - 1} - Q_n \|_{L^\frac{4d}{d+2}} + \| Q_n - Q \|_{L^\frac{4d}{d+2}} \\
					&	 \le\| Q_n^{p_n - 2} - 1 \|_{L^{\frac{12d}{6-d}}(B_T)} \| Q_n \|_{L^3(B_T)} + 
					\left(\| Q_n\|_{L^{\frac{4d(p_n-1)}{d+2}(B_T^c)}}^{p_n-1} + \|Q_n\|_{L^\frac{4d}{d+2}(B_T^c)}\right) +  \| Q_n - Q \|_{L^\frac{4d}{d+2}}
				\end{align*}
				From the $L^\infty$ bound (\ref{Linfbound}) and almost everywhere convergence (\ref{aeconv}), the dominant convergence implies that $\| Q_n^{p_n - 2} - 1 \|_{L^{\frac{12d}{6-d}}(B_T)} = o_n(1)$ for fixed $T$. We require $\d$ sufficiently small such that  
				\beq \frac{4d(p - 1)}{d+2} \in \left[\frac{11}{5}, 3\right]\quad \forall p \in  [2, 2+\d], \label{d2}\eeq 
				so (\ref{H1conv}) ensures that $\sup_n \left(\| Q_n\|_{L^{\frac{4d(p_n-1)}{d+2}(B_T^c)}}^{p_n-1} + \|Q_n\|_{L^\frac{4d}{d+2}(B_T^c)}\right) = o_{T\rightarrow \infty}(1)$. 
				Thus 
				\beq \| Q_n^{p_n - 1} - Q \|_{L^\frac{4d}{d+2}}  = o_n (1) \label{Linfsmall} \eeq
				Combining (\ref{convsmall}), (\ref{Linfsmall}) and assumption on $\phi_n$, we arrive at 
				\[|\uppercase\expandafter{\romannumeral1}| = o_n(1). \]
				
				For $\uppercase\expandafter{\romannumeral2}$, by H\"older inequality and Hardy-Littlewood-Sobolev inequality,
				\begin{align*}
					|\uppercase\expandafter{\romannumeral2}|& \lesssim_{d, \d} \| Q^{p_n} - Q^2 \|_{L^\frac{2d}{3d-2-2\a_n}} \|Q\|_{L^\frac{4d}{d+2}} \| \phi\|_{L^\frac{4d}{d+2}} \\
					&\le \| Q^{p_n-1} - Q \|_{L^\frac{4d}{d+2}} \|Q\|_{L^\frac{4d}{5d-6-4\a_n}} \|Q\|_{L^\frac{4d}{d+2}} \| \phi\|_{L^\frac{4d}{d+2}}.
				\end{align*}
				We further require $\d \ll 1$ such that 
				\beq \frac{4d}{5d-6-4\a_n} \in \left[\frac{11}{5}, 3\right]\quad \forall \a \in [d-2-\d, d-2+\d]. \label{d3}\eeq 
				(\ref{Linfbound}) and (\ref{aeconv}) indicate $\| Q\|_{L^\infty} \lesssim_{d, \d}$ 1, so we can prove $\|Q^{p_n -1} - Q\|_{L^\frac{4d}{d+2}} = o_{n \rightarrow \infty}(1)$ as before. These estimates imply
				\[ |\uppercase\expandafter{\romannumeral2}| = o_n(1). \]
				
				Regarding $\uppercase\expandafter{\romannumeral3}$, 
				we use dominant convergence. Obviously $\left((|\cdot|^{-\a_n} - |\cdot |^{-\a} )* Q^2 \right)Q \phi \rightarrow 0$ almost everywhere when $n \rightarrow \infty$. Also, since
				\begin{align*}
					\left|	\left((|\cdot|^{-\a_n} - |\cdot |^{-(d-2)} )* Q^2 \right)Q \phi \right|\le \left((|\cdot|^{-d+2+\d} + |\cdot |^{-d+2-\d} )* Q^2 \right)Q |\phi|, 
				\end{align*}
				and 
				\[ \int_{\real^d}  \left(|\cdot |^{-d+2\pm\d} * Q^2 \right)Q |\phi|dx \lesssim_{d, \d} \| Q \|^2_{L^{\frac{4d}{d+2\mp2\d}}} \| Q \|_{L^\frac{4d}{d+2}} \| \phi\|_{L^\frac{4d}{d+2}} \lesssim_{d, \d} 1 \]
				if 
				\beq \frac{11}{5} \le \frac{4d}{d+2\mp2\d} \le 3, \label{d4} \eeq 
				$\left((|\cdot|^{-d+2+\d} + |\cdot |^{-d+2-\d} )* Q^2 \right)Q |\phi|$ is a feasible dominant function to derive 
				\[ |\uppercase\expandafter{\romannumeral3}| = o_n(1). \]
				
				To conclude, if we take $\d$ small enough such that (\ref{d1}), (\ref{d2}), (\ref{d3}) and (\ref{d4}) hold, then the claim is true and thus the proposition is proven.
			\end{proof}
			
			\subsection{Non-degeneracy}\label{s52}
			
			In this subsection, we prove Theorem \ref{thmnondeg}. From Theorem \ref{mazhao} and the translation-invariance of $L_+$, we only need to discuss radial positive solutions with $(\a, p)$ close to $(d-2, 2)$. We will denote a radial positive solution of (\ref{Choquard}) with parameters $(d, \a, p)$ by  $Q_{d, \a, p}$ or $Q_{\a, p}$ if omitting $d$ causes no trouble.\footnote{We remark that our argument here does not require uniqueness of positive solutions for $(\a, p) \neq (d-2, 2)$.} 
			
			To begin with, we decompose $L_{+, u, d, \a, p}$ as (\ref{Lplus}) to be 
			\beq L_{+, u, d, \a, p} = -\Delta + 1 - (p-1)V_{u, d, \a, p} - p \calA_{u, d, \a, p} \eeq
			where (omitting $d$ for convenience \footnote{In particular, for $Q_{\a, p}$, we may further simplify the notation to be $L_{+, Q_{\a, p}}, V_{Q_{\a, p}}$ and $\calA_{Q_{\a, p}}$.})
			\begin{align}
				V_{u, \a, p}:=& (|\cdot|^{-\a} * |u|^p) |u|^{p-2}, \\
				\calA_{u, \a, p}\xi :=& (|\cdot|^{-\a} * (|u|^{p-2} u \xi))|u|^{p-2} u. 
			\end{align}
			In Appendix \ref{appC}, we discuss properties of these operators for varying $(u, \a, p)$: boundedness, compactness and continuous dependence on $(u, \a, p)$. They lay the foundation for the perturbative argument in proving Theorem \ref{thmnondeg} and Theorem \ref{thmuni}. 
			
			Consider the kernel of $L_{+, \a, p}$. From the non-degeneracy of Choquard equation at $( \a, p) = ( d-2, 2)$ for $d \in \{3, 4, 5\}$, we have 
			\beq \text{Ker}\, L_{+, Q_{d-2, 2}} = \text{span}\left\{\partial_i Q_{d-2, 2}\right\}_{i = 1}^d \label{specnew} \eeq
			Also by differentiating (\ref{Choquard}), we have 
			\beq \text{Ker}\, L_{+, Q{\a, p}} \supseteq \text{span}\left\{\partial_i Q_{\a, p}\right\}_{i = 1}^d \label{specsup} \eeq
			for any positive solution $Q_{\a,  p}$ of \eqref{Choquard} with $(\a, p)$ satisfying (\ref{para2}). 
			We will use an argument of spectral perturbation to show the other side of (\ref{specsup})
			\beq \text{dim}\, \text{Ker}\, L_{+, Q_{\a, p}} \le \text{dim}\,\text{Ker}\, L_{+, Q_{d-2, 2}} = d\label{specsub} \eeq
			when $(\a, p)$ close to $(d-2, 2)$. 
			
			From Lemma \ref{lemcpt} and $Q_{d-2, 2} \in W^{2, r}(\real^d)$ for $r \in (1, \infty)$,  $(-\Delta + 1)^{-1} V_{Q_{d-2, 2}}$ and $(-\Delta + 1)^{-1} \calA_{Q_{d-2, 2}}$ are compact operators on $L^2(\real^d)$ and thus $L_{+, Q_{\a, p}}$ is a compact perturbation of $-\Delta + 1$, $\sigma_{ess}(L_{+, Q_{\a, p}}) = \sigma_{ess}(-\Delta + 1) = [1, \infty)$. In particular, $0$ is an isolated eigenvalue of $L_{+, Q_{d-2, 2}}$. So we can define the Riesz projection
			\beq P_{0, Q_{d-2, 2}} := \frac{1}{2\pi i} \oint_{\partial D_r} (L_{+, Q_{d-2, 2}} - z)^{-1} dz, \eeq
			where $D_r := \{z \in \cpx: |z| < r \}$ and $r$ sufficiently small such that $\text{Im}\, P_{0, Q_{d-2, 2}} = \text{Ker}\, L_{+, Q_{d-2, 2}}$.
			
			Notice that 
			\begin{align*}
				L_{+, Q_{\a, p}} - L_{+, Q_{d-2, 2}} = -(p-1)(V_{Q_{\a, p}} - V_{Q_{d-2, 2}}) - p (\calA_{Q_{\a, p}} - \calA_{Q_{d-2, 2}}) -(p-2)(V_{Q_{d-2, 2}} + \calA_{Q_{d-2, 2}}).
			\end{align*}			
			Lemma \ref{lemcpt}, Lemma \ref{lemdiff} and the $H^1 \cap L^\infty$ approximation from Proposition \ref{propasym} imply
			\beq \| L_{+, Q_{\a, p}} - L_{+, Q_{d-2, 2}} \|_{L^2\rightarrow L^2} = o_{(\a, p)\rightarrow (d-2, 2)} (1). \eeq
			Thus for the $r$ taken as above and $\d_1 \le \d$ small enough, we have
			\[ \| (L_{+, Q_{\a, p}} - z)^{-1} \|_{L^2\rightarrow L^2} \le 2  \| (L_{+, Q_{d-2, 2}} - z)^{-1} \|_{L^2\rightarrow L^2} \]
			when $(\a, p) \in [d-2-\d_1, d-2+\d_1] \times [2, 2+\d_1]$ and thereafter
			\[ \| (L_{+, Q_{\a, p}} - z)^{-1} - (L_{+, Q_{d-2, 2}} - z)^{-1} \|_{L^2\rightarrow L^2} = o_{(\a, p)\rightarrow (d-2, 2)} (1).  \]
			So for such $(\a, p)$,
			\[ P_{0, Q_{\a, p}} :=  \frac{1}{2\pi i}  \oint_{\partial D_r} (L_{+, Q_{d-2, 2}} - z)^{-1} dz \]
			is a bounded operator on $L^2$ and 
			\[ \| P_{0, Q_{\a, p}} - P_{0, Q_{d-2, 2}} \|_{L^2\rightarrow L^2} =  o_{(\a, p)\rightarrow (d-2, 2)} (1). \]
			Note that $P_{0, Q_{d-2, 2}}$ is a Fredholm operator as a finite-rank projection. Via the theory of perturbation of Fredholm operators, there exists a $\d_2 \le \d_1$ such that (\ref{specsub}) holds for $(\a, p) \in [d-2-\d_2, d-2+\d_2], \times[2, 2+\d_2]$. This and (\ref{specsup}) concludes the proof of Theorem \ref{thmnondeg}.	    
			
			\subsection{Uniqueness} \label{s53} In this subsection, we prove Theorem \ref{thmuni}. We start with defining 
			\beq \mathbb{X}_d := L^2_{rad}(\real^d) \cap L^{10}_{rad}(\real^d),\quad d \in \{ 3, 4, 5\}. \eeq
			Note that $H^2(\real^d) \hookrightarrow \mathbb{X}_d$ for such $d$. 
			
			Now we can state and prove a local uniqueness result.
			
			\begin{prop}\label{proplocuni}
				For $d \in \{ 3, 4, 5\}.$ Let $Q_{d-2, 2}$ be the unique radial positive solution for $(\a, p) = (d-2, 2)$ of (\ref{Choquard}). Then there exist $\d_1 > 0$ and a $C^0$ map $\tilde{Q}: [d-2-\d_1, d-2+\d_1] \times [2, 2+\d_1] \rightarrow \mathbb{X}$ such that the following holds, where we denote $\tilde{Q}_{\a, p} := \tilde{Q}(\a, p)$.
				\begin{enumerate}[(1)]
					\item $\tilde{Q}_{\a, p}$ is an $H^1$ radial solution of (\ref{Choquard}) with parameters $(d, \a, p)$. 
					\item There exists $\e > 0$ such that $\tilde{Q}_{\a, p}$ is the unique $H^1$ radial solution of (\ref{Choquard}) with parameters $(d, \a, p)$ in the neighborhood $\{ u \in \mathbb{X}_d: \| u - Q_{d-2, 2}\|_{\mathbb X_d} \le \e \}$. In particular, $\tilde{Q}_{d-2, 2} = Q_{d-2, 2}$. 
				\end{enumerate}  
			\end{prop}
			\begin{proof}
				For $d \in \{ 3, 4, 5\}$ and $(d, \a, p)$ satifying (\ref{para2}), it's easy to see $u$ is an $H^1$ solution of (\ref{Choquard})if and only if $u \in \mathbb{X}_d$ is a solution of
				\beq u - (-\Delta + 1)^{-1} \left[ (|\cdot|^{-\a} * |u|^p) |u|^{p-2} u\right] = 0. \label{Choquard2} \eeq
				%				Indeed, Theorem \ref{thmprop} (3) indicates that every positive $H^1$ solution $u$ of (\ref{Choquard}) belongs to $H^2$, and thus $u \in \mathbb{X}$. Also (\ref{Choquard}) induces (\ref{Choquard2}). Conversely, we only need to verify $u \in H^1$ for $\mathbb{X}$ solution of (\ref{Choquard2}). From (\ref{para2}), we have $2 < \frac{d(2p-1)}{2d-\a} < 10$ for $d \in \{ 3, 4, 5\}$ and thus via (\ref{Choquard2}),
				%				\begin{align*} \| u \|_{H^2} &= \left\|(-\Delta + 1)^{-1} \left[(|\cdot|^{-\a} * |u|^p) |u|^{p-2} u\right]\right\|_{H^2}   \\
				%					&\lesssim_d \| (|\cdot|^{-\a} * |u|^p) |u|^{p-2} u  \|_{L^2} \lesssim_{d, \a, p} \| u \|_{L^\frac{d(2p-1)}{2d-\a}}^{2p-1} \le \| u \|_{\mathbb{X}}^{2p-1}. 
				%				\end{align*}
				Define $F: \mathbb{X}_d \times [d-2-\d_1, d-2+\d_1] \times [2, 2+\d_1] \rightarrow \mathbb{X}_d$ by 
				\[ F(u, \a, p) = u - (-\Delta + 1)^{-1} (|\cdot|^{-\a} * |u|^p) |u|^{p-2} u, \]  
				%        	where $B_\e(Q_{d-2, 2})$ is an $\e$-ball in $\mathbb{X}_d$ centered at $Q_{d-2, 2}$ with $\e$ and 
				where $\d_1$ is small enough and to be determined. First we require $\d_1$ to be smaller than the $\d$ in Lemma \ref{lemregf}, then $F$ is well-defined, continuous and differentiable w.r.t. $u$, and 
				\beq \partial_u F (u, \a, p) = \mathrm{Id} - (-\Delta + 1)^{-1} \left[ (p-1) V_{u, \a, p} + p \calA_{u, \a, p} \right] = (-\Delta + 1)^{-1} L_{+, u, \a, p} \label{Fu}\eeq
				is continuous at $(Q_{d-2, 2}, d-2, 2)$. If $\partial_u F(Q_{d-2, 2}, d-2, 2)$ is invertible in $L(\mathbb{X}_d)$, then these facts enable us to apply the implicit function theorem Proposition \ref{propift}.  The assertions (1) and (2) follow directly from its conclusion and the above equivalence. 
				
				To conclude the proof, we check the invertibility of $\partial_u F$ at $(Q_{d-2, 2}, d-2, 2)$. Again from Lemma \ref{lemcpt} and regularity of $Q_{d-2, 2}$, $(-\Delta + 1)^{-1} V_{Q_{d-2, 2}}$ and $(-\Delta + 1)^{-1} \calA_{Q_{d-2, 2}}$ are compact operators on $L^2(\real^d)$ and $L^{10}(\real^d)$ respectively, and thus also compact on $\mathbb{X}_d$. This indicates that $\partial_u F (Q_{d-2, 2}, d-2, 2)$ is a Fredholm operator on $\mathbb{X}_d$ by \eqref{Fu}. On the other hand, non-degeneracy of $L_{+, Q_{d-2, 2}}$ indicates that 
				\[ \mathrm{Ker}\,L_{+, Q_{d-2, 2}}\big|_{L^2_{rad}(\real^d)} = \{ 0\}, \]
				which implies that $\partial_u F(Q_{d-2, 2}, d-2, 2)$ is injective on $\mathbb{X}_d$. Properties of Fredholm operators show that $\partial_u F(Q_{d-2,2},d-2, 2)$ is also bijective and therefore has a bounded inverse. That finishes the proof.
			\end{proof}
			
			Now we are in place to prove Theorem \ref{thmuni}.
			
			\begin{proof}[Proof of Theorem \ref{thmuni}]From Theorem \ref{mazhao}, we only need to show the uniqueness of the $H^1$ radial positive solution. 
				Given $d \in \{ 3, 4, 5\}$, take $\d_1$ and $\e$ as in Proposition \ref{proplocuni}. Proposition \ref{propasym} indicates that there exists $\d_2 > 0$ such that any $H^1$ radial positive solution $Q_{\a, p}$ for (\ref{Choquard}) with parameters $(\a, p) \in [d-2-\d_2, d-2+\d_2] \times [2, 2+\d_2]$ satisfying
				\beq \| Q_{\a, p} - Q_{d-2, 2}\|_{\mathbb{X}} \le \e. \label{close} \eeq
				Now taking $\d = \min\{\d_1, \d_2\}$, for every $(\a, p) \in [d-2-\d, d-2+\d] \times [2, 2+\d]$, Theorem \ref{thmprop} (1) indicates that there exists an $H^1$ radial positive solution $Q_{\a, p}$ satisfying \eqref{close}, which, by Proposition \ref{proplocuni}, must equal to $\tilde{Q}_{\a, p}$ and be unique in this neighborhood. Since Proposition \ref{propasym} also guarantees the non-existence of $H^1$ radial positive solutions outside this neighborhood, $Q_{\a, p}$ is exactly the unique $H^1$ radial positive solution for (\ref{Choquard}).
			\end{proof}
			
			\cu{Acknowledgement.} The author thanks Baoping Liu and Tao Zhou for many helpful discussions. A large part of this work was done during the author's staying in Peking University, which was supported by the the NSF of China (No. 12071010, 11631002). Also, special thanks go to to some anonymous passenger of Beijing Subway Line 2 who returned the author's lost backpack in July 2021. Without him this work would have suffered a lot.
			
			\appendix
			\section{A computational lemma}
			
			We derive the following lemma estimating the integration of exponential function multiplied by a polynomial. 
			\begin{lem} \label{lemexp}For $\a \in \real$, $\beta \ge \frac{1}{2}$ and $R \ge 1$, we define
				\beq I(R;\a, \beta) := \int_R^\infty  r^{-\a} e^{-\beta r} dr. \eeq 
				Then we have 
				\beq I(R;\a, \beta) \sim_{\a, \beta} R^{-\a} e^{-\beta R}, \label{expest}\eeq
				and the constant can be taken uniformly for $(\a, \b) $ in a compact subset of $\real \times [\frac{1}{2}, \infty)$. 
			\end{lem}
			\begin{proof}
				By changing of variables $t := \beta r$,
				\[ I(R; \a, \b) = \int_{\beta R}^\infty (\beta^{-1}r)^{-\a} e^{-t} \beta^{-1}dt = \beta^{\a-1} I(\b R; \a, 1).   \]
				Thus we only need to prove (\ref{expest}) for $\b = 1$ and $R \ge \frac{1}{2}$. 
				
				The upper bound for $\a \ge 0$ comes easily as
				\[ I(R; \a, 1) \le R^{-\a} \int_R^\infty e^{-r}dr = R^{-\a}e^{-R}. \]
				Similarly, the lower bound for $\a \le 0$ holds.  
				
				As for the lower bound for $\a  > 0$, it's a classical result for $\a = n \in \posint_{+}$ (refer to \cite[5.1.19]{abramowitz1964handbook}) 
				\[ I(R; n, 1) \ge e^{-R}R^{-n + 1} \frac{1}{R+n} \ge \frac{1}{2n+1} R^{-n}e^{-R},\quad R \ge \frac{1}{2},\,\, n= 1, 2, 3\ldots \]
				And for the interrange case $\a \in (n-1, n)$ with $n \in \posint_+$, 
				\[ I(R; \a, 1) \ge \int_R^\infty R^{n-\a}r^{-n}e^{-r} dr \ge \frac{1}{2\a+3} R^{-\a} e^{-R}. \]
				
				Finally, we check the upper bound for $\a < 0$ case. Using integration by parts
				\[
				I(R; \a, 1) = R^{-\a}e^{-R} - \a I(R; \a+1, 1), \quad \a \neq 0, \]
				$I(R; \a, 1)$ with $-\a \in (n-1, n]$ can be bounded within $n$ times of iterations
				\begin{align*} 
					I(R; \a, 1) &= R^{-\a}e^{-R} + \sum_{k=1}^{n-1} \left[\prod_{j=1}^k (-\a -1 + j)\right] R^{-\a - k}e^{-R} \\
					&\,\,\,\,+ \left[\prod_{j=1}^n (-\a -1 + j)\right] I(R; \a + n, 1) \\
					&\lesssim_{\a} R^{-\a} e^{-R}. 
				\end{align*}
				The final inequality utilizes the upper bound for $\a \ge 0$ case and $R \ge \frac{1}{2}$.
			\end{proof}
			
			\section{A refined implicit function theorem}
			
			We carefully check the proof of the classical implicit function theorem in Banach space (see for example \cite[Theorem 1.2.1]{chang2006methods}) to slightly relax the $C^1$ condition to $C^1$ at one point. It will be used in proving Proposition \ref{proplocuni}. In this subsection, we use $B_r(x)$, $\bar{B}_r(x)$ to denote open and closed balls respectively in general Banach spaces.
			
			\begin{prop}\label{propift}
				Let $X, Y, Z$ be Banach spaces, $U \subset X \times Y$ be an open set. Suppose that $f \in C(\bar{U}, Z)$ is differentiable w.r.t. $y$. For a point $(x_0, y_0) \in \bar{U}$, if $f_y: \bar{U} \rightarrow L(Y, Z)$ is continuous at $(x_0, y_0)$ and
				\begin{align*}
					f(x_0, y_0) &= 0,\\
					f_y^{-1}(x_0, y_0) &\in L(Z, Y),
				\end{align*}
				then there exist $r, r_1 > 0$ and a $C^0$ map $u: \bar{B}_r(x_0) \rightarrow \bar{B}_{r_1}(y_0)$, such that 
				\[ \left\{\begin{array}{l} 
					\bar{B}_r(x_0) \times \bar{B}_{r_1}(y_0) \subset \bar{U},\\
					u(x_0) = y_0,  \\
					f(x, u(x)) = 0 \quad \forall \,\, x \in \bar{B}_r(x_0),
				\end{array}\right. \]
				Furthermore, if $f(x, y) = 0$ for some $(x, y) \in \bar{B}_r(x_0) \times \bar{B}_{r_1}(y_0)$, then $y = u(x)$. 
			\end{prop}
			\begin{proof}
				Consider 
				\[ g(x, y):= f_y^{-1}(x_0, y_0) \circ f(x+x_0, y+y_0).\]
				We look for the solution $y = u(x) \in \bar{B}_{r_1}(0)$ of $g(x, y) = 0$ for $x \in \bar{B}_r(0)$, with $r, r_1$ small enough and determined later. Define 
				\[ R(x, y):= y - g(x, y),\]
				and then we will check that $R(x, \cdot)$ for $x\in \bar{B}_r(0)$ is a contraction mapping on $\bar{B}_{r_1}(0)$.
				
				Firstly, for any $x \in \bar{B}_r(0)$ and $y_1, y_2 \in \bar{B}_{r_1}(0)$, we have
				\begin{align*} &\| R(x, y_1) - R(x, y_2) \|_Y = \| y_1 - y_2 - [g(x, y_1) - g(x, y_2)] \|_Y \\ 
					\le& \int_0^1 \| \text{id}_{Y} - f_y^{-1}(x_0, y_0) \circ f_y (x + x_0, ty_1 + (1-t)y_2 + y_0) \|_{L(Y, Y)} dt \cdot \| y_1 - y_2 \|_Y 
				\end{align*}
				From the continuity of $f_y$ at $(x_0, y_0)$, there exist $r, r_1 \ll 1$ such that for $(x, y) \in \bar{B}_r(x_0) \times \bar{B}_{r_1}(y_0) \subset \bar{U}$,  
				\[       	\| f_y(x, y) - f_y(x_0, y_0) \|_{L(Y, Z)} \le \frac{1}{2}\|f_y^{-1}(x_0, y_0)\|_{L(Z, Y)}^{-1}.\]
				This leads to 
				\beq  \| R(x, y_1) - R(x, y_2) \|_Y  \le \frac{1}{2} \| y_1 - y_2 \|_Y.\label{contract}\eeq
				
				Secondly, we verify $R(x, \cdot): \bar{B}_{r_1} (0) \rightarrow \bar{B}_{r_1}(0)$. Indeed, note that
				\begin{align*}
					\| R(x, y)\|_{Y} &\le \|R(x, 0)\|_Y + \| R(x, y) - R(x, 0) \|_Y \\
					&\le \| f_y^{-1}(x_0, y_0) \|_{L(Z, Y)} \| f(x+x_0, y_0) \|_Z + \frac{1}{2} \| y \|_{Y}. 
				\end{align*}
				Using continuity of $f$, we shrink $r$ further to guarantee 
				\[ \sup_{x \in \bar{B}_r(0)}\| f(x+x_0, y_0) \|_Z \le \frac{r_1}{2} \| f_y^{-1}(x_0, y_0) \|_{L(Z, Y)}^{-1}, \]
				which leads to $R(x, y) \le r_1$ for $(x, y ) \in \bar{B}_{r}(0) \times \bar{B}_{r_1}(0)$. 
				
				Thus $R$ is a contraction and for all $x \in \bar{B}_r(0)$, there exists a unique $y \in \bar{B}_{r_1}(0)$ satisfying $g(x, y) = 0$, namely $f(x+x_0, y+y_0) = 0$. We denote by $v(x)$ the solution $y$, and $u(x) := v(x-x_0) + y_0$. 
				
				Finally we prove the continuity of $u: \bar{B}_r(x_0) \rightarrow \bar{B}_{r_1}(y_0)$, which is equivalent to that of $v: \bar{B}_r(0) \rightarrow \bar{B}_{r_1}(0)$. Using \eqref{contract}, for $x, x' \in \bar{B_r}(0)$
				\begin{align*}
					\| v(x) - v(x')\|_Y &= \| R(x, u(x)) - R(x', u(x')) \|_Y \\
					&	\le \frac{1}{2} \| v(x) - v(x') \|_Y + \| R(x, u(x)) - R(x', u(x)) \|_Y.
				\end{align*}
				So we obtain
				\[ 	\| v(x) - v(x')\|_Y \le 2 \| R(x, u(x)) - R(x', u(x)) \|_Y. \]
				The continuity of $R$ implies that of $v$. 
			\end{proof}

			\section{Properties of $V, \calA$ and Regularity of $F$} \label{appC}
			
			In this section, we discuss boundedness, compactness and continuous dependence of linear operators $V_{u, d, \a, p}$ and $\calA_{u, d, \a, p}$. They are related to the linearized operator $L_{+, Q_{\a, p}, \a, p}$ for non-degeneracy \S \ref{s52}, the $\mathbb X_d$-valued function $F$ in \S \ref{s53} and an improvement of regularity argument in \S \ref{s51}. We recall their definitions from \S \ref{s52}:
			\begin{align*}
				V_{u, d, \a, p}:=& (|\cdot|^{-\a} * |u|^p) |u|^{p-2}, \\
				\calA_{u, d, \a, p}\xi :=& (|\cdot|^{-\a} * (|u|^{p-2} u \xi))|u|^{p-2} u.
			\end{align*}
			
			For simplicity of notation, we denote the region of the $(\a, p)$ we consider to be 
			\beq \Omega_{d, \d}:=[d-2-\d, d-2 + \d] \times [2, 2+\d] \eeq with $\delta \ll 1$. Obviously $(\a, p) \in \Omega_{d, \d}$ satisfies \eqref{para2}. 
			
			One estimate (and its perturbation) will be frequently used 
			\beq \| (|\cdot|^{-{d-2}} * f)g \|_{L^2} \lesssim_{d} \| f \|_{L^\frac{2d}{d-2}} \| g \|_{L^d}. \label{2dd} \eeq
			%			{\color{red} Save this? We remark that the estimates, especially considering those exponents, in this section are far from sharp, but just enough for our purpose.}
			
			\subsection{For $u \in \mathbb{X}_d$} 
			
			First we consider $u \in \mathbb{X}_d := L^2_{rad}(\real^d) \cap L^{10}_{rad}(\real^d)$ for $d \in \{ 3, 4, 5\}$.
			
			\begin{lem}\label{lembdd}For $d \in \{3, 4, 5\}$, $(\a, p) \in \Omega_{d, \d}$ and $\d \ll 1$. If $u \in\mathbb{X}_d$, $V_{u, \a, p}$ and $\calA_{u, \a, p}$ are bounded from $\mathbb{X}_d$ to $L^2_{rad}$.  
			\end{lem} 
			
			\begin{proof}
				Using \eqref{2dd}, 
				\begin{align*} \| V_{u, \a, p} f \|_{L^2} & \le \| |\cdot|^{-\a} * |u|^{p} \|_{L^{\frac{2d}{d-2}}} \| |u|^{p-2}f\|_{L^d} \\
					&\lesssim_{d, \a, p}  \| u \|_{L^\frac{2dp}{3d-2\a-2}}^{p} \| u \|_{L^{d(p-1)}}^{p-2} \| f\|_{L^{d(p-1)}}  \lesssim_{\| u \|_{\mathbb{X}_d}} \|f \|_{\mathbb{X}_d}, \\
					\| \calA_{u, \a, p} f \|_{L^2} & \le \| |\cdot|^{-\a} * |u|^{p-2}uf \|_{L^{\frac{2d}{d-2}}} \| |u|^{p}\|_{L^d} \\
					&\lesssim_{d, \a, p} \| u \|_{L^\frac{2dp}{3d-2\a-2}}^{p-1} \| f \|_{L^\frac{2dp}{3d-2\a-2}}  \| u \|_{L^{d(p-1)}}^{p-1} 
					\lesssim_{\| u \|_{\mathbb{X}_d}} \|f \|_{\mathbb{X}_d},
				\end{align*}
				where we need $\d $ not large so that for any 
				\beq \frac{2dp}{3d-2\a-2}, d(p-1) \in [2, 10], \quad \forall\,\, d\in \{ 3, 4, 5\}, \,\, (\a, p) \in \Omega_{d, \d} \label{reqd1}\eeq
				%				From the definition of $\Omega_d$, we have $2 < \frac{d(2p-1)}{2d-\a} < \frac{d(2p-1)}{(d-2)p} < 6$ for $d \in \{ 3, 4, 5\}$ and thus
				%				\begin{align*} \| V_{u, \a, p} f \|_{L^2} &=  \| (|\cdot|^{-\a} * |u|^p) |u|^{p-2} f  \|_{L^2} \lesssim_{d, \a, p} \| u \|_{L^\frac{d(2p-1)}{2d-\a}}^{2p-2} \| f \|_{L^\frac{d(2p-1)}{2d-\a}} \lesssim_{\| u \|_{\mathbb{X}_d}} \|f \|_{\mathbb{X}_d}, \\
				%					\| \calA_{u, \a, p} f \|_{L^2} &=  \| (|\cdot|^{-\a} * |u|^{p-2}uf) |u|^{p-2} u  \|_{L^2} \lesssim_{d, \a, p} \| u \|_{L^\frac{d(2p-1)}{2d-\a}}^{2p-2} \| f \|_{L^\frac{d(2p-1)}{2d-\a}} \lesssim_{\| u \|_{\mathbb{X}_d}} \|f \|_{\mathbb{X}_d}. \\
				%				\end{align*}
			\end{proof}
			
			\begin{lem}\label{lemcont} For $d \in \{ 3, 4, 5\}$, $\d \ll 1$, the maps $V: \mathbb{X} \times \Omega_{d, \d} \rightarrow L(\mathbb{X}_d, L^2)$ and  $\calA: \mathbb{X} \times \Omega_{d, \d} \rightarrow L(\mathbb{X}_d, L^2)$ are well-defined. The following statements are true.
				\begin{enumerate}[(1)]
					%	    		\item For general $u \in \mathbb{X}$, ({\color{red} Definition}) $V_{u, \a, p}$ and $\calA_{u, \a, p}$ are continuous with respect to $(u, \a, p) \in \mathbb{X} \times [d-2-\d, d-2+\d] \times (2, 2+\d]$ under the norm topology for linear bounded operators on $L^2$. Also, it is not continuous for general $u \in \mathbb{X}$ at $p = 2$. 
					%	    		\item For any $\nu > 0$, there exists $\e > 0$ such that for  $(u_1, \a_1, p_1), (u_2, \a_2, p_2) \in \bar{B}_\e(Q_{d-2, 2}) \times [d-2-\d, d-2+\d] \times [2, 2+\d]$ where $\bar{B}_\e(Q_{d-2, 2})$ is the ball in $\mathbb{X}$,
					%	    		\beq  \| V_{u_1, \a_1, p_1} - V_{u_2, \a_2, p_2} \|_{L^2 \rightarrow L^2} +\| \calA_{u_1, \a_1, p_1} - \calA_{u_2, \a_2, p_2} \|_{L^2 \rightarrow L^2} \le \nu.  \eeq
					\item $\calA$ is continuous on $ \mathbb{X} \times \Omega_{d, \d}$. 
					\item $V$ is continuous on $\mathbb{X} \times \{ (\a, p) \in \Omega_{d, \d}: p > 2\}$. 
					\item $V$ is continuous at $(Q_{\a, 2}, \a, 2)$ and discontinuous at $(Q_{\a, 2}\varphi_R, \a, 2)$ for any $|\a-(d-2)| \le \d$ and $R > 0$. Here $Q_{\a, 2}$ is a positive and radially decreasing solution for (\ref{Choquard}) with parameters $(d, \a, 2)$, and $\varphi_R$ is a smooth cutoff compactly supported in $B_{2R}$ and equals $1$ on $B_R$.
					
					%	    		 Namely, for any $\nu > 0$, there exists $\e > 0$ such that for  $(u, \a, p) \in \bar{B}_\e(Q_{d-2, 2}) \times [d-2-\d, d-2+\d] \times [2, 2+\d]$ where $\bar{B}_\e(Q_{d-2, 2})$ is the ball in $\mathbb{X}$, 
					%	    		\beq  \| V_{u, \a, p} - V_{Q_{d-2, 2}} \|_{L^2 \rightarrow L^2} +\| \calA_{u, \a, p} - \calA_{Q_{d-2, 2}} \|_{L^2 \rightarrow L^2} \le \nu.  \eeq
				\end{enumerate}
			\end{lem}
			
			\begin{proof} We fisrt let $\d$ be required by Lemma \ref{lembdd}, namely \eqref{reqd1}. We may put on further requirements during the proof.
				
				(1)  Given $(u, \a, p) \in \mathbb{X}_d \times \Omega_{d, \d}$, We show that for any $\e > 0$, there exists $\d_1 > 0$ such that for $(u_1, \a_1, p_1) \in \mathbb{X}_d \times \Omega_{d, \d}$ and $\| u - u_1\|_{\mathbb{X}_d} + |\a - \a_1| + |p-p_1| < \d_1$,
				\begin{align*}
					\| \calA_{u, \a, p} f - \calA_{u_1, \a_1, p_1} f \|_{L^2} \le \e  \| f \|_{\mathbb{X}_d} .
				\end{align*}
				We will take $\d$ small enough to ensure
				\begin{align}
					\| \calA_{u, \a, p} f - \calA_{u, \a, p_1} f \|_{L^2}& \le \frac{\e}{3} \|f\|_{\mathbb{X}_d}, \label{cala1} \\
					\| \calA_{u, \a, p_1} f - \calA_{u, \a_1, p_1} f \|_{L^2}& \le \frac{\e}{3} \|f\|_{\mathbb{X}_d},\label{cala2}\\
					\| \calA_{u, \a_1, p_1} f - \calA_{u_1, \a_1, p_1} f \|_{L^2}& \le \frac{\e}{3} \|f\|_{\mathbb{X}_d}. \label{cala3}
				\end{align}
				And they conclude the proof of (1).

				Firstly
				\begin{align}
					&\| \calA_{u, \a, p} f - \calA_{u, \a, p_1} f \|_{L^2} \nonumber \\
					\le & \| \left[ |\cdot|^{-\a} * (|u|^{p-2} uf - |u|^{p_1 - 2} uf)  \right] |u|^{p_1 - 2} u \|_{L^2} \label{term11} \\
					+ & \| \left[ |\cdot|^{-\a} * (|u|^{p-2} uf)  \right] (|u|^{p_1 - 2}u - |u|^{p-2} u) \|_{L^2} \label{term12}
				\end{align}
				These two terms are estimated in a similar way, so we only do \eqref{term11} as an example. If $p_1 = p$, then there is nothing we need to do. Assume $p_1 > p$ (the case $p < p_1$ can be treated similarly). We take $M > 1$ and partition the range of $|u|$ to get
				%				\beq \left| |u|^{p-2}(r) - |u|^{p_1-2}(r) \right| \le C(\d, M) |p-p_1| |u|^{p-2}(r),\quad\forall r \in \{1/M \le  |u| \le M\}, \label{M1}\eeq
				%				and 
				\beq \left| |u|^{p-2}(r) - |u|^{p_1-2}(r) \right| \le \left\{\begin{array}{ll} 2|u|^{p_1 - 2}(r)&   r \in \{ |u| > M\}, \\ 
					C(M) |p-p_1| |u|^{p-2}(r)& r \in \{1/M \le  |u| \le M\}, \\
					2|u|^{p - 2}(r)&   r \in \{ |u| < 1/M\}.
				\end{array} \right. \label{M2}\eeq
				We still use the exponent as Lemma \ref{lembdd}
				\begin{align*}
					&\quad\,\,	\| \left[ |\cdot|^{-\a} * (|u|^{p-2} uf - |u|^{p_1 - 2} uf)  \right] |u|^{p_1 - 2} u \|_{L^2} \\
					&	\lesssim_{d, \a, p} \| |u|^{p-2} u - |u|^{p_1 - 2} u \|_{L^\frac{2dp}{(3d-2\a-2)(p-1)}} \|f\|_{L^\frac{2dp}{3d-2\a-2}} \| u \|_{L^{d(p_1-1)}}^{p_1-1}\\
					&	\le \left[2\| u\|_{L^{\frac{2dp}{3d-2\a-2} }(\{|u| < 1/M \}) }^{p-1} + 2\| u\|_{L^{\frac{(2dp)(p_1 - 1)}{(3d-2\a-2)(p-1)} }  (\{|u| > M \}) }^{p_1 -1} + C( M) |p_1 - p| \| u\|_{L^{\frac{2dp}{3d-2\a-2}}}^{p-1}  \right] \| u \|_{\mathbb{X}_d }^{p_1-1} \| f \|_{\mathbb{X}_d} \\
					& \le \left[ 2 \|u\|_{\mathbb{X}_d (\{|u| < 1/M \})}^{p-1} + 2\|u\|_{\mathbb{X}_d (\{|u| > M \})}^{p_1-1} + C(M)|p_1 - p| \|u\|_{\mathbb{X}_d }^{p-1} \right] \| u \|_{\mathbb{X}_d }^{p_1-1} \| f \|_{\mathbb{X}_d},
				\end{align*}
				where we also require $\d_1$ small such that $\frac{(2dp)(p_1 - 1)}{(3d-2\a-2)(p-1)} \in [2, 10]$. 
				Now first take $M \gg 1$ depending on $u$ such that $\|u\|_{\mathbb{X}_d (\{|u| < 1/M \})} + \|u\|_{\mathbb{X}_d (\{|u| > M \})} \ll \e (\|u\|_{\mathbb{X}_d}^{1+\d} + 1)^{-1}$, and then require $\d \ll 1$ to make $C(M)|p_1 - p| \ll \e \| u \|_{\mathbb{X}_d}^{-(p+p_1 - 2)}$. We see (\ref{term11}) can be bounded by $\frac{\e}{6} \| f \|_{\mathbb{X}_d}$. (\ref{term12}) comes in a similar way. So (\ref{cala1}) is confirmed.
				
				Next, consider (\ref{cala2}), the variation of $\a$. 
				%				We restrict $\d \le \d_0 \ll 1$ such that $0 < \a - 2\d_0 < \a + 2\d_0 < d$ and $\frac{d(2p_1-1)}{2d-\a\pm 2\d_0} \in [2, 10]$. 
				Take a $N \gg 1$, then
				\begin{align*}
					\left| |x|^{-\a} - |x|^{-\a_1}\right| \le \left\{ \begin{array}{ll} 
						|x|^{-\a-\d_0} + |x|^{-\a + \d_0} & |x| < 1/N \,\,\mathrm{or}\,\, |x| > N, \\
						C(N) |\a - \a_1| \left(|x|^{-\a-\d_0} + |x|^{-\a + \d_0}\right) & 1/N \le |x| \le N.  \end{array} \right.
				\end{align*}
				so we easily see 
				\[ \| |x|^{-\a} - |x|^{-\a_1} \|_{L^\frac{d}{d-2-2\d} + L^\frac{d}{d-2+2\d}} = o_{|\a - \a_1| \to 0}(1). \]
				Hence by Young's inequality and similar estimates as \eqref{2dd}, 
				\begin{align*}
					\| \calA_{u, \a, p_1} f - \calA_{u, \a_1, p_1} f \|_{L^2}& = \|[( |\cdot|^{-\a} - |\cdot|^{-\a_1} ) * |u|^{p_1-2}uf ] |u|^{p_1-2} u \|_{L^2} \\
					&	\lesssim_{d, \d} \| |u|^{p_1 -2} u f \|_{L^\frac{2d}{d+2+ 4\d} \cap L^\frac{2d}{d+2- 4\d} } \| |u|^{p_1 -2} u \|_{L^d} \cdot o_{|\a - \a_1|\to 0} (1)\\
					& \le \|u\|_{\mathbb{X}_d}^{2p_1-2} \| f\|_{\mathbb{X}_d} \cdot o_{|\a - \a_1|\to 0} (1) 
				\end{align*}
				where we require $\d$ small such that $\frac{2d}{d+2\pm 4\d} \ge 1 $ and $\frac{2dp}{d+2\pm 4\d} \in [2, 10]$ for any $p \in [2, 2+\d]$. Then \eqref{cala2} follows this and $\d_1$ small enough. 
				
				Finally, for the variation of $u$, we have
				\begin{align*}
					&\quad\,\,\| \calA_{u, \a_1, p_1} f - \calA_{u_1, \a_1, p_1} f \|_{L^2} \\
					&\le 
					\|[ |\cdot|^{-\a_1} * (|u|^{p_1 - 2} u f - |u_1|^{p_1 -2} u_1 f) ] |u_1|^{p_1 -2} u_1 \|_{L^2} \\
					& \quad\,\,+  \| [ |\cdot|^{-\a_1} * (|u|^{p_1 - 2} u f ) ] (|u|^{p_1 - 2} u - |u_1|^{p_1 -2} u_1 ) \|_{L^2} \\
					& \lesssim_{d, \a, p} \| f \|_{L^{\frac{2dp}{3d-2\a -2}}} \left(\| u \|_{L^{\frac{2dp}{3d-2\a -2}}}^{p_1 - 2} + \| u_1 \|_{L^{\frac{2dp}{3d-2\a -2}}}^{p_1 - 2} \right) \| u - u_1 \|_{L^{\frac{2dp}{3d-2\a -2}}} \| u_1 \|_{L^{d(p-1)}}^{p_1 -1 } \\
					& \quad \,\, + \| f \|_{L^{\frac{2dp}{3d-2\a -2}}} \| u \|_{L^{\frac{2dp}{3d-2\a -2}}}^{p_1 - 1} \left(\| u \|_{L^{d(p-1)}}^{p_1 - 2} + \| u_1 \|_{L^{d(p-1)}}^{p_1 - 2} \right) \| u - u_1 \|_{L^{d(p-1)}}
				\end{align*}
				The last inequality comes from pointwise estimate 
				\[\left| |x|^{p_1 -2}x - |y|^{p_1-2}y\right| \lesssim_{p_1} (|x|^{p-2} + |y|^{p-2}) |x-y| \]
				and nonlinear estimate as Lemma \ref{lembdd}. Taking $\d_1$ small enough and we obtain the last control (\ref{cala3}).
				
				(2) Similarly, we will show that for $(u, \a, p) \in \mathbb{X}_d \times \Omega_{d, \d}$ and every $\e > 0$, there exists $\d_1 > 0$ such that for $(u_1, \a_1, p_1) \in \mathbb{X}_d \times \Omega_{d, \d}$ and $\| u - u_1 \|_{\mathbb{X}_d} + | \a - \a_1 | + |p-p_1 | < \d_1$, we have
				\begin{align}
					\| V_{u, \a, p} f - V_{u, \a, p_1} f \|_{L^2} &\le \frac{\e}{3} \|f\|_{\mathbb{X}_d}, \label{v1}\\
					\| V_{u, \a, p_1} f - V_{u, \a_1, p_1} f \|_{L^2} &\le \frac{\e}{3} \|f\|_{\mathbb{X}_d}, \label{v2}\\
					\| V_{u, \a_1, p_1} f - V_{u_1, \a_1, p_1} f \|_{L^2} &\le \frac{\e}{3} \|f\|_{\mathbb{X}_d}. \label{v3}
				\end{align} 
				It is easy to check that (\ref{v2}) and (\ref{v3}) follows almost the same estimates as (\ref{cala2}) and (\ref{cala3}) respectively, which also works when $p_1 = 2$. 
				
				For (\ref{v1}),
				\beq \begin{split}
					&\| V_{u, \a, p} f - V_{u, \a, p_1} f \|_{L^2}  \\
					\le & \| \left[ |\cdot|^{-\a} * (|u|^{p} - |u|^{p_1})  \right] |u|^{p_1 - 2} f \|_{L^2}  
					+ \| \left[ |\cdot|^{-\a} * |u|^{p}  \right] (|u|^{p - 2} - |u|^{p_1-2}) f \|_{L^2}.
				\end{split}\label{v12} \eeq
				Like (1), we assume $0 < p_1 -p \ll 1$ and partition the range of $|u|$ with respect to $M > 1$ and $1/M$. Then using (\ref{M2}), the first term follows in a similar way as (1)
				\begin{align*}
					&\quad\,\,	 \| \left[ |\cdot|^{-\a} * (|u|^{p} - |u|^{p_1})  \right] |u|^{p_1 - 2} f \|_{L^2} \\
					&	\lesssim_{d, \a, p} \| |u|^{p}  - |u|^{p_1} \|_{L^\frac{2d}{3d-2\a -2}} \|f\|_{L^{d(p-1)} } \| |u|^{p_1-2} \|_{L^{\frac{d(p-1)}{p_1 - 2}}} \\
					& \le \left[ 2 \|u\|_{\mathbb{X}_d (\{|u| < 1/M \})}^{p} + 2\|u\|_{\mathbb{X}_d (\{|u| > M \})}^{p_1} + C(p,M)|p_1 - p| \|u\|_{\mathbb{X}_d }^{p} \right] \| u \|_{\mathbb{X}_d }^{p_1-2} \| f \|_{\mathbb{X}_d}.
				\end{align*}
				We remark that this estimate also works when $p=2$ or $p_1 = 2$, since $\| |u|^{p_1-2}\|_{L^\frac{d(2p-1)}{(2d-\a)(p_1 - 2)}} = \|1\|_{L^\infty} = 1$. For the second term,
				\begin{align*}
					&\quad\,\,	\| \left[ |\cdot|^{-\a} * |u|^{p}  \right] (|u|^{p - 2} - |u|^{p_1-2}) f \|_{L^2} \\
					&	\lesssim_{d, \a, p} \| |u|^{p} \|_{L^\frac{2d}{3d-2\a -2}} \bigg( \|f\|_{L^{d(p-1)} } \| |u|^{p_1-2} - |u|^{p-2} \|_{L^\frac{d(p-1)}{p-2}(\{ |u| \le M\}) }  \\
					&\quad + \|f\|_{L^{d(p_1-1)}} \| |u|^{p_1-2} - |u|^{p-2} \|_{L^\frac{d(p_1-1)}{p_1-2}(\{ |u| > M\}) }  \bigg)\\
					& \le \| u \|_{\mathbb{X}_d }^{p} \| f \|_{\mathbb{X}_d}  \bigg[ 2\| |u|^{p - 2} \|_{L^\frac{d(p-1)}{p-2}(\{|u| < 1/M\})} + 2\| |u|^{p_1 - 2} \|_{L^\frac{d(p_1-1)}{p_1-2}(\{|u| > M\})} \\
					& \quad +  C(p, M) |p-p_1|  \| |u|^{p-2} \|_{L^\frac{d(p-1)}{p-2}} \bigg]\\
					& \le\| u \|_{\mathbb{X}_d }^{p} \| f \|_{\mathbb{X}_d}\left[ 2 \|u\|_{\mathbb{X}_d (\{|u| < 1/M \})}^{p-2} + 2\|u\|_{\mathbb{X}_d (\{|u| > M \})}^{p_1-2} + C(p,M)|p_1 - p| \|u\|_{\mathbb{X}_d }^{p-2} \right].
				\end{align*}
				Then $p, p_1 > 2$ is necessary to get the smallness from $\| u \|_{\mathbb{X}_d(\{ |u| < 1/M\})} + \| u \|_{\mathbb{X}_d(\{ |u| > M\})}  = o_{M\rightarrow \infty}(1)$. So we can take $M \gg 1$ and then $\d_1 \ll 1$ (so that $p_1-2 $ and $p-2$ have a positive lower bound) to guarantee (\ref{v1}). 
				
				(3) To discuss the continuity at $(Q_{\a, 2}, \a, 2)$ or $(Q_{\a, 2}\varphi_R, \a, 2)$, we still consider another $(u_1, \a_1, p_1) \in \mathbb{X}_d \times \Omega_d$ and the three parts as (\ref{v1})-(\ref{v3}). As in (2), (\ref{v2}) and (\ref{v3}) hold all these cases. What distinguishes the cases is (\ref{v1}).
				
				To be more specific, the trouble is the second term $\| \left[ |\cdot|^{-\a} * |u|^{2}  \right] (1 - |u|^{p_1-2}) f \|_{L^2}$
				in the further partition (\ref{v12}). When $u = Q_{\a, 2}$, we can use its strict positivity, $L^\infty$-bounded and radially decreasing to derive smallness. Let $\{Q_{\a, 2} < 1/M \}:= B_{R(M)}^c$, then the crux is that $R(M) \rightarrow \infty$ as $M \rightarrow \infty$. Taking $M > \| Q_{\a, 2} \|_{L^\infty}$, we have
				\begin{align*}
					&\quad\,\,\| \left[ |\cdot|^{-\a} * Q_{\a, 2}^{2}  \right] (1 - Q_{\a, 2}^{p_1-2}) f \|_{L^2} \\
					& \le \| \left[ |\cdot|^{-\a} *Q_{\a, 2}^{2}  \right] (1 - Q_{\a, 2}^{p_1-2}) f \|_{L^2(\{ Q_{\a, 2} < 1/M\})} + \| \left[ |\cdot|^{-\a} * Q_{\a, 2}^{2}  \right] (1 - Q_{\a, 2}^{p_1-2}) f \|_{L^2 ( \{Q_{\a, 2} \ge 1/M\})} \\
					& \le \| |\cdot|^{\a} * Q_{\a, 2}^2 \|_{L^\infty(B_{R(M)}^c)} \| f \|_{L^2} + \| |\cdot|^{-\a} * Q_{\a, 2}^2 \|_{L^\infty} \| f \|_{L^2} C(M) |p_1-2|.
				\end{align*}
				Using Proposition \ref{propnonlocal1}, 
				\begin{align*}
					\| |\cdot|^{-\a} * Q_{\a, 2}^2 \|_{L^\infty} &\lesssim_{d, \a} \int_0^\infty Q_{\a, 2}^2 (r) r^{d-1-\a} dr \lesssim_{d, \a} \| Q_{\a, 2} \|_{L^\infty \cap L^2}^2 < \infty \\
					\| |\cdot|^{-\a} * Q_{\a, 2}^2 \|_{L^\infty(B_{R(M)}^c)}&\lesssim_{d, \a} R(M)^{-\a} \int_0^\infty Q_{\a, 2}^2 (r) r^{d-1} dr \lesssim_{d} R(M)^{-\a} \| Q_{\a, 2} \|_{L^2}^2.
				\end{align*}
				So $M \gg 1$ and then $\d \ll 1$ will ensure the smallness of $\| \left[ |\cdot|^{-\a} * Q_{\a, 2}^{2}  \right] (1 - Q_{\a, 2}^{p_1-2}) f \|_{L^2}$ and thereafter (\ref{v1}) is verified. That is the continuity at $(Q_{\a, 2}, \a, 2)$ for $\a \in (0, d)$.
				
				Regarding the discontinuity for $u = Q_{\a, 2} \varphi_R$, $R > 0$, we claim that
				\beq \| V_{Q_{\a, 2}\varphi_R, \a ,2} - V_{Q_{\a, 2}\varphi_R, \a ,p_1} \|_{L(\mathbb{X}_d, L^2)} \ge \| V_{Q_{\a, 2}\varphi_R, \a ,2} \chi_{B_{2R}^c}  \|_{L(\mathbb{X}, L^2)} > 0 \eeq
				for any $p_1 > 2$. 
				Indeed, for any $f$ supported on $B_{2R}^c$, 
				\begin{align*}
					\| V_{Q_{\a, 2}\varphi_R, \a ,2} f- V_{Q_{\a, 2}\varphi_R, \a ,p_1} f\|_{L^2} &= \| \left[ |\cdot|^{-\a} * (Q_{\a, 2} \varphi_R)^{2}  \right] (1 -  (Q_{\a, 2} \varphi_R)^{p_1 -2}) f \|_{L^2} \\
					&\ge \| \left[ |\cdot|^{-\a} * (Q_{\a, 2} \varphi_R)^{2}  \right] f \|_{L^2}.
				\end{align*}
			\end{proof}
			
			Now we apply the results above to derive regularity of $F$ in \S \ref{s51}. In particular, the discontinuity will also appear and thus hinder the application of the common version of the implicit function theorem. We first recall the definition of $F$:
			\begin{align*}
				F(u, \a, p) &= u - (-\Delta + 1)^{-1} (|\cdot|^{-\a} * |u|^p) |u|^{p-2} u.
				%		    	 \\
				%			 F_u (u, \a, p) &= \mathrm{Id} - (-\Delta + 1)^{-1} \left[ (p-1) V_{u, \a, p} + p \calA_{u, \a, p} \right] = (-\Delta + 1)^{-1} L_{+, u, \a, p}.
			\end{align*}

			\begin{lem}\label{lemregf}
				There exists $\d > 0$ such that $F:  \mathbb{X}_d \times \Omega_{d, \d} \rightarrow \mathbb{X}_d$ is well-defined, continuous w.r.t. $(u, \a, p)$ and differentiable w.r.t. $u$; $F_u: \mathbb{X}_d \times \Omega_{d, \d} \rightarrow L(\mathbb{X}_d)$ is continuous at $(Q_{d-2, 2}, d-2, 2)$. Besides, $\partial_u F$ is discontinuous on $ \{u \in \mathbb{X}_d : \| u - Q_{d-2, 2}\|_{\mathbb X_d} \le \e \} \times \Omega_{d, \e} $ for any $\e > 0$. 
				%    	For any $\nu$, there exists $\e > 0$ such that $F: B_\e(Q_{d-2, 2}) \times [d-2-\d, d-2+\d] \times [2, 2+\d] \rightarrow \mathbb{X}$ is $C^0$ and satisfies 
				%    	\[ \| F_u (u, \a, p) - F_u(Q_{d-2, 2}, d-2, 2) \|_{\mathbb{X} \rightarrow 
				%    	\mathbb{X}} \le \nu \] 
				%    for all $(u, \a, p) \in B_\e(Q_{d-2, 2}) \times [d-2-\d_1, d-2+\d_1] \times [2, 2+\d]$.
				%    	Here $B_\e(Q_{d-2, 2})$ the $\e$-ball in $\mathbb{X}_d$ centered at $Q_{d-2, 2}$.
				%    
				%    	Besides, $F_u : B_\e(Q_{d-2, 2}) \times [d-2-\e, d-2+\e] \times [2, 2+\e] \rightarrow L(\mathbb{X}_d)$ is not continuous for any $\e > 0$. 
			\end{lem}
			
			\begin{proof}
				We will frequently use the following nonlinear estimate given by \eqref{2dd}:
				\beq\| \left(|\cdot|^{-\a} * (|f_1|^{p-2} f_2 f_3)\right) |f_4|^{p-2} f_5 \|_{L^2} \lesssim_{d, \a, p} \| f_1 \|_{\mathbb{X}_d}^{p-2} \| f_2 \|_{\mathbb{X}_d} \| f_3 \|_{\mathbb{X}_d} \| f_4 \|_{\mathbb{X}_d}^{p-2} \| f_5 \|_{\mathbb{X}_d}. \label{nonest}
				\eeq
				
				(1) Firstly, we check that $F$ is well-defined. Using Sobolev embedding $H^2(\real^d) \hookrightarrow \mathbb{X}_d$ and (\ref{nonest}), we easily see
				\[ \|F(u, \a, p)\|_{\mathbb{X}_d} \lesssim_{d} \| u \|_{\mathbb{X}_d} + \|V_{u, \a, p} u \|_{L^2} \lesssim_{d, \a, p}  \| u \|_{\mathbb{X}_d} +  \| u \|_{\mathbb{X}_d}^{2p-1}. \]
				
				(2) Next we prove the continuity of $F$. Note that
				\[ \| F(u, \a, p) - F(u_1, \a_1, p_1)\|_{\mathbb X_d} \lesssim_d \| V_{u, \a, p} u - V_{u_1, \a_1, p_1} u_1 \|_{L^2}.\]
				We only need to show the following estimate
				\beq \| V_{u, \a, p} u - V_{u, \a, p_1} u \|_{L^2} = o_{|p-p_1| \to 0}(1) \label{v1'} \eeq
				and apply \eqref{v2} and \eqref{v3}. Due to the existence of $u$, the left hand side is exactly the same as $\calA_{u, \a, p} u - \calA_{u, \a, p_1} u$, and thus this estimate follows \eqref{cala1}.

				(3) Then we turn to the Fr\'echet differentiability of $F$ w.r.t. $u$. We claim that
				\beq \partial_u F(u, \a, p) = \mathrm{Id} - (-\Delta + 1)^{-1} \left[ (p-1) V_{u, \a, p} + p \calA_{u, \a, p} \right]. \eeq
				We need to prove that for the $\partial_u F$ defined above, for any $h \in \mathbb{X}_d$, 
				\beq \left\|F(u+h, \a, p) - F(u, \a, p) - \partial_u F(u, \a, p)h\right\|_{\mathbb{X}_d} = o_{\| h \|_{\mathbb{X}_d} \to 0}(\|h\|_{\mathbb{X}_d}). \label{diffF} \eeq
				Again, from $H^2(\real^d) \hookrightarrow \mathbb{X}_d$ and a direct computation
				\begin{align*}
					&\left\|F(u+h, \a, p) - F(u, \a, p) - \partial_u F(u, \a, p)h\right\|_{\mathbb{X}_d} \\
					\lesssim_{d, \a, p}& \left\| \left[|\cdot|^{-\a} * (|u+h|^p - |u|^p - p |u|^{p-2}uh ) \right] |u+h|^{p-2} (u+h) \right\|_{L^2} \\
					+&\left\| \left(|\cdot|^{-\a} * p |u|^{p-2}uh \right)(|u+h|^{p-2} (u+h) - |u|^{p-2} u ) \right\|_{L^2} \\
					+& \left\| \left(|\cdot|^{-\a} *  |u|^p\right) [|u+h|^{p-2} (u+h) - |u|^{p-2} u - (p-1)|u|^{p-2} h] \right\|_{L^2},
				\end{align*} 
				the estimate (\ref{diffF}) follows elementary pointwise estimates
				\begin{align*}
					\left| |u+h|^p - |u|^p - p |u|^{p-2}uh \right| &\lesssim_p (|u|^{p-2} + |h|^{p-2} ) |h|^2 \\
					\left| |u+h|^{p-2} (u+h) - |u|^{p-2} u \right| &\lesssim_p (|u|^{p-2} + |h|^{p-2} ) |h| \\
					\left| |u+h|^{p-2} (u+h) - |u|^{p-2} u - (p-1)|u|^{p-2}h \right| &\lesssim_p \left\{ \begin{array}{ll} (|u|^{p-3} + |h|^{p-3} ) |h|^2 & p > 3 \\ |h|^{p-1} & p \in [2, 3]\end{array} \right.        	
				\end{align*}
				and (\ref{nonest}). Note that $\partial_u F(u, \a, p) \in L(\mathbb{X}_d)$ comes directly from Sobolev embedding and Lemma \ref{lembdd}.
				
				(4) Next we show the continuity of $\partial_u F$ at $(Q_{d-2, 2}, d-2, 2)$. Indeed
				\begin{equation}\begin{split}
						&\| \partial_u F (Q_{d-2, 2}, d-2, 2) - \partial_u F(u, \a, p) \|_{L(\mathbb{X}_d)} \\ \lesssim_{d}& (p-2) \| V_{Q_{d-2, 2}} \|_{L(\mathbb{X}_d, L^2)} + (p-2) \| \calA_{Q_{d-2, 2}} \|_{L(\mathbb{X}_d, L^2)} \\
						+&(p-1) \| V_{Q_{d-2, 2}} - V_{u, \a, p} \|_{L(\mathbb{X}_d, L^2)} +  p \| \calA_{Q_{d-2, 2}} - \calA_{u, \a, p} \|_{L(\mathbb{X}_d, L^2)}. 
					\end{split} \label{contest}
				\end{equation} 
				So Lemma \ref{lemcont} (1)(3) and Lemma \ref{lembdd} imply the continuity. 
				
				(5) Finally, the discontinuity again follows from the discontinuity argument in Lemma (\ref{lemcont}) (3). For every $\e > 0$, there exists $R \gg 1$ such that $\|Q_{d-2, 2} \varphi_R - Q_{d-2, 2} \|_{\mathbb{X}_d} < \e $ and $\partial_u F$ is discontinuous at $(Q_{d-2, 2} \varphi_R, d-2, 2)$, where  $\varphi_R$ was defined in Lemma \ref{lemcont}. Indeed, for any $p > 2$, 
				\begin{align*}
					&\| \partial_u F (Q_{d-2, 2} \varphi_R, d-2, 2) - \partial_u F(Q_{d-2, 2} \varphi_R, d-2, p) \|_{L(\mathbb{X}_d)} \\ 
					\ge & \|(-\Delta + 1)^{-1} (V_{Q_{d-2, 2} \varphi_R, d-2, 2} - V_{Q_{d-2, 2} \varphi_R, d-2, p} )  \|_{L(\mathbb{X}_d)} 
					- C(d) \bigg(  (p-2) \| V_{Q_{d-2, 2} \varphi_R, d-2, 2} \|_{L(\mathbb{X}_d, L^2)} \\
					& + (p-2) \| \calA_{Q_{d-2, 2} \varphi_R, d-2, 2} \|_{L(\mathbb{X}_d, L^2)} 
					+  p \| \calA_{Q_{d-2, 2} \varphi_R, d-2, 2} - \calA_{Q_{d-2, 2} \varphi_R, d-2, p} \|_{L(\mathbb{X}_d, L^2)} \bigg).
				\end{align*}
				Considering the support of $(Q_{d-2, 2}\varphi_R)^{p-2}$, we see the first term is lower bounded by $ \| (-\Delta + 1)^{-1}\circ (V_{Q_{d-2, 2}\varphi_R, d-2 ,2} \chi_{B_{2R}^c}) \|_{\mathbb{X}_d} > 0$. And the negative part is $o_{p \rightarrow 2}(1)$ due to the boundedness and continuity from Lemma \ref{lembdd} and Lemma \ref{lemcont} (1). So this estimate indicates the discontinuity. 
			\end{proof}
			
			\subsection{For $u = Q_{\a, p}$.} Now we discuss the boundness, compactness and continuity of $V_{u, \a, p}$ and $\calA_{u, \a, p}$ for $u$ with better bound than $\mathbb X_d$. In the main text, $u$ will be taken as $Q_{\a, p}$. Besides, here we will not restrict to spaces of radial functions.
			
			\begin{lem}\label{lemcpt}
				For $d \in \{3, 4, 5\}$, $(\a, p) \in \Omega_d$ and $q \in (1, \infty)$. For $u \in L^2 \cap L^\infty$, we have 
				\begin{enumerate}
					\item $V_{u, \a, p}: L^q \to L^q$ is bounded.
					\item If in addition $u \in C^0$,  then $V_{u, \a, p} : W^{1,q} \rightarrow L^q$ is compact.
					\item $\calA_{u, \a, p} :L^q \rightarrow L^q$ is bounded.
					\item If in addition $u \in W^{1, r}$ for $r \in (1, \infty)$, then $\calA_{u, \a, p} :L^q \rightarrow L^q$ is compact.
				\end{enumerate}
			\end{lem}
			
			\begin{proof}
				(1) Using H\"older inequality, we estimate the $L^\infty$ norm of $V_{u, \a, p}$ as a function
				\begin{align*} 
					\| V_{u, \a, p}\|_{L^\infty} & \le \| |\cdot|^{-\a} * |u|^p \|_{L^\infty} \| |u|^{p-2} \|_{L^\infty} \\
					& \le \| |\cdot|^{-\a} \|_{L^1 + L^\infty} \| |u|^p\|_{L^1\cap L^\infty} \| u \|_{L^\infty}^{p-2} \lesssim_{d, \a, p, \| u \|_{L^2 \cap L^\infty}} 1
				\end{align*}
				This immediately implies that $V_{u, \a, p} : L^q \to L^q$ is bounded for any $q \in (1, \infty)$. 
				
				(2) Moreover, when $u \in C^0$, we claim that the function $V_{u, \a, p}$ vanishes at infinity and is uniformly continuous: for every $\e > 0$, there exists $\d > 0$ such that 
				\beq | V_{u, \a, p}(x) - V_{u, \a, p}(y) | \le \e,\quad \forall\,\, |x-y| \le \d. \label{unicont}  \eeq
				Indeed, the vanishing comes from $u \in L^\infty$ and $|\cdot|^{-\a} * |u|^p$ decays at infinity since $|\cdot|^{-\a} \in L^1 + L^\infty$ and $|u|^p \in L^1 \cap L^\infty$. The uniform continuity of $V_{u, \a, p}$ then follows its local continuity which comes from $u \in C^0$.
				
				Now we prove the compactness via Fr\'echet-Kolmogorov compactness theorem. Let $\{f_n\}_n \subset W^{1, q}$ be a bounded sequence. Then $\{V_{u, \a, p} f_n\}$ is uniformly bounded in $L^q$ from the argument above. We need to show $\{V_{u, \a, p}f_n\}_n$ is equicontinuous and uniformly localized in $L^q$. Indeed,
				\begin{itemize} 
					\item Equicontinuous: For $h \in \real^d$, (\ref{unicont}) indicates
					\begin{align*} 
						&\| (V_{u, \a, p}f_n)(\cdot+ h) - V_{u, \a, p}f_n \|_{L^q} \\
						\le & \| (V_{u, \a, p}(\cdot + h) - V_{u, \a, p})f_n(\cdot + h) \|_{L^q} + \| V_{u, \a, p} (f_n(\cdot+ h) - f_n) \|_{L^q}  \\
						\le & \| f_n \|_{L^q} \cdot o_{h\rightarrow 0}(1) + \| V_{u, \a, p}\|_{L^\infty} \| f_n \|_{W^{1, q}} |h| = o_{h\rightarrow 0}(1) 
					\end{align*}
					\item Uniform localization: since $V_{u, \a, p}$ vanishes at infinity,
					\[\|V_{u, \a, p}f_n\|_{L^q(B_R^c)} \le \| V_{u, \a, p} \|_{L^\infty(B_R^c) } \| f_n\|_{L^q} = o_{R\rightarrow \infty}(1). \]
				\end{itemize}
				So Fr\'echet-Kolmogorov compactness theorem verifies the precompactness. 
				
				(3) We divide two cases for boundedness. When $q \in (1, \frac{2d}{d-\a})$, boundedness follows
				\beq\begin{split}
					\| \calA_{u, \a, p} f \|_{L^q} &\le \left\| |\cdot|^{-\a} * |u|^{p-2}uf \right\|_{L^{\left(\frac{1}{q} - \frac{d-\a}{2d}\right)^{-1}}} \left\| |u|^{p-1} \right\|_{L^\frac{2d}{d-\a}} \\
					&\lesssim_{d, \a} \|f\|_{L^q}\left\| |u|^{p-1} \right\|_{L^\frac{2d}{d-\a}}^2 \lesssim_{d, \a, p} \| f \|_{L^q}.
				\end{split}\label{nonest1}\eeq
				And when $q \ge \frac{2d}{d-\a} > 2$, we estimate as
				\beq\begin{split}
					\| \calA_{u, \a, p} f \|_{L^q} &\le \left\| |\cdot|^{-\a} * |u|^{p-2}uf \right\|_{L^{2q}} \left\| |u|^{p-1} \right\|_{L^{2q}} \\
					&\lesssim_{d, \a} \|f\|_{L^{q}} \left\| |u|^{p-1} \right\|_{L^{\left(\frac{d-\a}{d} - \frac{1}{2q}\right)^{-1}}} \left\| |u|^{p-1} \right\|_{L^{2q}} \lesssim_{d, \a, p} \| f \|_{L^q}.
				\end{split} \label{nonest2}\eeq
				
				(4) Again, take a bounded sequence $\{f_n\}_n \subset L^q$. We verify the equicontinuity and uniform localization of $\{\calA_{u, \a, p}f_n\}_n$ to confirm compactness of $\calA_{u, \a, p}$. 
				\begin{itemize} 
					\item Equicontinuous: We first prove a pointwise estimate: for any $x, h \in \real^d$, $\a \in (0, d) $,
					\beq \left| |x+h|^{-\a} - |x|^{-\a}  \right| \lesssim_{d, \a} |h|^{\b(d, \a)} \left(  |x+h|^{-\a-\b(d, \a)} + |x|^{-\a-\b(d, a)} \right),\label{difference} \eeq
					where $$\b(d, \a) = \min\{1, \frac{d-\a}{2} \}.$$ Indeed, when $|x| \ge 2 |h|$, we have
					\begin{align*}
						\left| |x+h|^{-\a} - |x|^{-\a}  \right| &\le \left| \int_0^1 \frac{1}{-\a} |x+th|^{-\a-2}(x+th) \cdot h dt \right| \\
						&\le \a^{-1} \left(\frac{1}{2}|x|\right)^{-\a - 1} |h| \lesssim_{d, \a} |x|^{-\a-\b} |h|^{\b}.
					\end{align*}
					And when $|x| < 2 |h|$, we have $\max\{|x|, |x+h|\} \le 3 |h|$, so 
					\begin{align*}
						\left| |x+h|^{-\a} - |x|^{-\a}  \right| &\le  |x+h|^{-\a} + |x|^{-\a}  \le \left( |x+h|^{-\a-\b} + |x|^{-\a-\b} \right) |3h|^{\b}. 
					\end{align*}
					
					For any $h \in \real^d$,
					\begin{align}
						&\left\| (\calA_{u, \a, p} f_n)(\cdot + h) - \calA_{u, \a, p}f_n \right\|_{L^q} \nonumber\\
						\le & \left\| [(|\cdot|^{-\a} * (|u|^{p-2}uf_n))(\cdot + h) - (|\cdot|^{-\a} * (|u|^{p-2}uf_n))(\cdot)] (|u|^{p-2}u)(\cdot+h) \right\|_{L^q} \label{second}\\
						+& \left\| (|\cdot|^{-\a} * (|u|^{p-2}uf_n)) [(|u|^{p-2}u)(\cdot+h) - (|u|^{p-2}u)] \right\|_{L^q}. \label{third}
					\end{align}
					We will show that (\ref{second}) and (\ref{third}) are $o_{h\rightarrow 0}(1)$ uniformly for $f_n$ bounded in $L^q$, which verifies equicontinuity. From
					\begin{align*}
						\left\| (|u|^{p-2}u)(\cdot+h) - (|u|^{p-2}u)\right\|_{L^r} \lesssim_{p} \| u\|_{L^{\infty}}^{p-2} \|\nabla u \|_{L^r} |h|, \quad r \in ( 1, \infty),
					\end{align*}
					the term (\ref{third}) is estimated as (\ref{nonest}) and (\ref{nonest2}) to be $o_{h \rightarrow 0}(1)$. And from (\ref{difference}), 
					\begin{align*} &\left|(|\cdot|^{-\a} * (|u|^{p-2}uf_n))(x + h) - (|\cdot|^{-\a} * (|u|^{p-2}uf_n))(x)\right| \\
						\lesssim_{d, \a}& |h|^{\b} \left[ \left|(|\cdot|^{-\a-\b} * (|u|^{p-2}uf_n))(x + h)\right| + \left|(|\cdot|^{-\a-\b} * (|u|^{p-2}uf_n))(x)\right|\right].
					\end{align*}
					Note that $\a + \b \in (0, d)$ and $\b > 0$, we can similarly bound (\ref{second}) to be $o_{h\rightarrow 0}(1)$. 
					
					\item Uniform localization: Estimate as (\ref{nonest1}) and (\ref{nonest2}),
					\begin{align*}
						\| \calA_{u, \a, p} f_n\|_{L^q(B_R^c)} &= \| (|\cdot|^{-\a} * (|u|^{p-2}uf_n)) |u\chi_{B_R^c}|^{p-1}\|_{L^q}  \\
						& \lesssim_{d, \a} \left\{\begin{array}{ll}
							\| f_n \|_{L^q} \| |u|^{p-1} \|_{L^\frac{2d}{d-\a}} \| |u|^{p-1} \|_{L^\frac{2d}{d-\a}(B_R^c)}& q \in (1, \frac{2d}{d-\a}), \\ 
							\| f_n \|_{L^q} \| |u|^{p-1} \|_{L^{\left(\frac{d-\a}{d} - \frac{1}{2q} \right)^{-1}}} \| |u|^{p-1}\|_{L^{2q} (B_R^c)} &  q \in [\frac{2d}{d-\a},\infty)
						\end{array} \right.\\
						&  = o_{R\rightarrow \infty}(1).
					\end{align*}
				\end{itemize}
			\end{proof}
			
			The following miscellaneous estimates will be used in proving Proposition \ref{propasym} and Theorem \ref{thmnondeg} in \S \ref{s4}.
			
			\begin{lem} \label{lemdiff} For $d \in \{ 3, 4, 5\}$, there exists $\d >0$ such that the following statements hold. For $(\a, p) \in \Omega_{d, \d}$, $u_0, u \in L^2 \cap L^\infty$ and $u_0$ positive and radially decreasing, we have
				\begin{align}
					\| V_{u_0, d-2, 2} u_0 - V_{u, \a, p} u \|_{L^{\frac{2d}{d-2}}} &\lesssim_{d, \d, u_0} o_{|\a - (d-2)| + |p-2|\to 0} (1) + \| u_0 - u \|_{L^2 \cap L^{\frac{2d}{d-2}}}   \label{contineq1}\\
					\| V_{u_0, d-2, 2} - V_{u, \a, p} \|_{L^\infty} &\lesssim_{d, \d, u_0} o_{|\a - (d-2)| + |p-2| + \| u - u_0 \|_{L^2 \cap L^\infty}\to 0} (1) \label{contineq2}\\
					\| \calA_{u_0, d-2, 2} - \calA_{u, \a, p} \|_{L^2 \to L^2} &\lesssim_{d, \d, u_0} o_{|\a - (d-2)| + |p-2| \to 0} (1) + \| u - u_0 \|_{L^2 \cap L^\infty}. \label{contineq3}
				\end{align} 
			\end{lem}
			
			\begin{proof}
				The proof of these inequalities are like those in Lemma \ref{lemcont}, with main difference comes from the choice of norms since we have better control of $u_0$ and $u$ this time. So we only sketch the proof. 
				
				(1) Proof of \eqref{contineq1}. It suffices to show the following two estimates:
				\begin{align}
					\| V_{u_0, d-2, 2} - V_{u_0, d-2, p}\|_{L^{\infty}} &\lesssim_{d, \d, u_0}  o_{|p-2|\to 0} (1) \label{contineq11} \\
					\| V_{u_0, d-2, p} - V_{u_0, \a, p}\|_{L^{\infty}} &\lesssim_{d, \d} \| u_0 \|_{L^2 \cap L^\infty}^{2p-2}\cdot   o_{|\a - (d-2)|\to 0} (1) \label{contineq12} \\
					\| V_{u_0, \a, p} u_0 - V_{u, \a, p} u \|_{L^{\frac{2d}{d-2}}} &\lesssim_{d, \d} \| u_0 - u \|_{L^2 \cap L^{\frac{2d}{d-2}}} \left( \| u_0 \|_{L^2 \cap L^\infty}^{2p-1} + \| u\|_{L^2 \cap L^\infty}^{2p-1} \right)  \label{contineq13}
				\end{align}
				
				Inequality \eqref{contineq11} is exactly included in the proof of \eqref{v1} in Lemma \ref{lemcont} (3) (where we utilize the positive and radially decreasing of $u$). For \eqref{contineq12}, we adjust the exponents in proving \eqref{v2}:
				\begin{align*}
					&\| V_{u_0, d-2, p} - V_{u_0, \a, p}\|_{L^{\infty}} \le \| |u_0|^{p-2}\|_{L^\infty} \left\| (|\cdot|^{-(d-2)} - |\cdot|^{-\a}) * |u_0|^p\right\|_{L^\infty} \\
					& \quad \le \| u_0\|_{L^\infty}^{p-2} \| |\cdot|^{-(d-2)} - |\cdot|^{-\a} \|_{L^1 + L^\infty} \| u_0\|_{L^p \cap L^\infty}^p \le \| u_0 \|_{L^2 \cap L^\infty}^{2p-2}\cdot   o_{|\a - (d-2)| \to 0} (1)
				\end{align*}
				And for \eqref{contineq13}, it can be deduced easily from H\"older and Hardy-Littlewood-Sobolev estimates like \eqref{cala3}.
				
				(2) Proof of \eqref{contineq1}. Given \eqref{contineq11} and \eqref{contineq12}, we only need to estimate 
				\[ \|  V_{u_0, \a, p} - V_{u, \a, p} \|_{L^\infty} \le \| (|\cdot|^{-\a} * (|u_0|^p - |u|^p)) |u_0|^{p-2} \|_{L^\infty} + \|  (|\cdot|^{-\a} * |u|^p) ( |u_0|^{p-2} - |u|^{p-2}) \|_{L^\infty}. \]
				The only tricky part is the second term, where we take an $M > \|u_0\|_{L^\infty}$ and let $R(M)$ be such that $B_{R(M)} = \{ u \ge M^{-1}\}$. Then when $\| u_0 - u\|_{L^\infty}$ is bounded, we have
				\[\|  (|\cdot|^{-\a} * |u|^p) ( |u_0|^{p-2} - |u|^{p-2}) \|_{L^\infty(B_{R(M)}^c)} = o_{M \to \infty}(1) \]
				due to the vanishing of $|\cdot|^{-\a} * |u|^p$ at infinity. And suppose $\| u_0 - u\|_{L^\infty} \le M^{-2}$, we can estimate the $L^\infty(B_{R(M)})$ part using
				\begin{align*} | |u_0|^{p-2} - |u|^{p-2} | 
					%                	= | u_0 |^{p-2} \left| 1 - \left(\frac{|u|}{|u_0|}\right)^{p-2} \right| 
					= | u_0 |^{p-2} \left| 1 - \left( 1 - \frac{|u_0| - |u|}{|u_0|}\right)^{p-2} \right| 
					\lesssim_\d \| u_0 \|_{L^\infty}^{p-2}  |p-2|\frac{|u_0 - u|}{M^{-1}}.
				\end{align*}
				
				(3) Proof of \eqref{contineq3}. This is almost the same as \eqref{cala1}-\eqref{cala3} using perturbation of \eqref{2dd}.
			\end{proof}
			
			\bibliographystyle{siam}
			\bibliography{Bib-1}
			
		\end{document}